\documentclass[11pt,a4paper]{amsart}
\usepackage[DIV14,paper=a4,BCOR14mm,headinclude]{typearea}
\usepackage[british]{babel}
\usepackage{amssymb}
\usepackage{amsmath}
\usepackage{graphicx}
\usepackage{epstopdf}
\usepackage{color}
\usepackage{tikz}
\usetikzlibrary{patterns}
\usepackage{diagbox}
\usepackage{pgfplots}
\usepackage{pgfplotstable}
\usepackage{enumitem}
\usepackage{hyperref}
\usepackage{fp-eval}
\usepackage{esint}
\usepackage{subfigure}
\usepackage{caption}
\usepackage{mathabx}
\usepackage{booktabs}
\usepackage{subdepth}
\usepackage{lineno}
\definecolor{darkblue}{rgb}{0,0,.7}
\hypersetup{colorlinks=true,linkcolor=darkblue,citecolor=darkblue}

\newlist{alphenum}{enumerate}{1}
\setlist[alphenum]{fullwidth,label={(\alph*)}}

\newtheorem{theorem}{Theorem}[section]

\newtheorem{remark}[theorem]{Remark}
\newtheorem{lemma}[theorem]{Lemma}

\newtheorem{assumption}[theorem]{Assumption}
\numberwithin{figure}{section}
\numberwithin{table}{section}
\numberwithin{equation}{section}

\newcommand{\ds}{\, ds}

\newcommand{\abs}[1]{\left \lvert #1 \right \rvert}
\newcommand{\jump}[1]{ [\![ {#1} ]\!]}
\newcommand{\average}[1]{\{\!\{ #1  \}\!\}}

\newcommand{\vecb}[1]{\boldsymbol{#1}}

\begin{document}

\title[Inf-sup stabilized Scott--Vogelius for Navier--Stokes]{Inf-sup stabilized Scott--Vogelius pairs on general simplicial grids for Navier--Stokes equations}

\author{Naveed Ahmed}
\address{Gulf University for Science and Technology, Block 5, Building 1, Muarak Al-Abdullah Area, West Mishref Kuwait and
Weierstrass Institute for Applied Analysis and Stochastics (WIAS), Mohrenstr.
39, 10117 Berlin, Germany}
\email{ahmed.n@gust.edu.kw}
\author{Volker John}
\address{Weierstrass Institute for Applied Analysis and Stochastics (WIAS), Mohrenstr.
39, 10117 Berlin, Germany and
Freie Universit\"at Berlin, Department of Mathematics and
Computer Science, Arnimallee 6, 14195 Berlin, Germany}
\email{john@wias-berlin.de, ORCID 0000-0002-2711-4409}
\author{Xu Li}
\address{School of Mathematics, Shandong University, Jinan 250100, China}
\email{xulisdu@126.com}
\author{Christian Merdon}
\address{Weierstrass Institute for Applied Analysis and Stochastics (WIAS), Mohrenstr.
39, 10117 Berlin, Germany}
\email{merdon@wias-berlin.de}

\begin{abstract}
  This paper considers the discretization of the time-dependent
  Navier--Stokes equations with the family of inf-sup
  stabilized Scott--Vogelius pairs recently introduced in
  [John/Li/Merdon/Rui, arXiv:2206.01242, 2022] for the Stokes problem.
  Therein, the velocity space is obtained by enriching the $\vecb{H}^1$-conforming Lagrange element space
  with some $\vecb{H}(\mathrm{div})$-conforming Raviart--Thomas functions,
  such that the divergence constraint is satisfied exactly.
  In these methods arbitrary shape-regular simplicial grids can be used.

  In the present paper two alternatives for discretizing the convective terms are considered.
  One variant leads to a scheme that still only involves volume integrals, and
  the other variant employs upwinding known from DG schemes. Both variants ensure the conservation of linear momentum and angular momentum in
  some suitable sense. In addition, a pressure-robust and convection-robust velocity error estimate is derived, i.e.,
  the velocity error bound does not depend on the pressure and
  the constant in the error bound for the kinetic energy does not blow up for small viscosity.
  After condensation of the enrichment unknowns and all non-constant
  pressure unknowns, the method can be reduced to a $\vecb{P}_k-P_0$-like system for arbitrary velocity
  polynomial degree $k$. Numerical
  studies verify the theoretical findings.
\end{abstract}

\subjclass[2020]{76D05, 76M10, 65M60}

\keywords{Navier--Stokes equations,
finite element methods,
divergence-free,
pressure-robust,
convection-robust,
a priori bounds}

\maketitle

\section{Introduction}

Incompressible flows are modeled by the transient Navier--Stokes equations
and seek a velocity $\vecb{u}$ and pressure $p$ such that
\begin{align}\label{eq:navierstokes}
\begin{array}{rcll}
\partial_t \vecb{u}+(\vecb{u}\cdot\nabla)\vecb{u}-\nu \Delta \vecb{u} +\nabla p & = & \vecb{f} & \mbox{in } (0, T]\times\Omega,\\
\mathrm{div} (\vecb{u}) & = & 0 &  \mbox{in } (0, T]\times\Omega,\\
\vecb{u}(0)&=&\vecb{u}^0 & \mbox{in } \Omega,\\
\vecb{u} & = & \vecb{0} & \mbox{on }(0, T]\times\partial \Omega,
\end{array}
\end{align}
in a bounded Lipschitz domain
$\Omega \subset \mathbb R^d$, $d\in\{2,3\}$, and for a given $T<\infty$.
The given data $\nu\in\mathbb R$, $\vecb{f}$, and
$\vecb{u}^0$ denote the dimensionless viscosity, the external force,
and the initial velocity, respectively.
Note that problem \eqref{eq:navierstokes} is already given in a dimensionless form.
For simplicity, it is assumed that $\vecb{f} \in L^2(0,T;\vecb{L}^2(\Omega))$.
The method under consideration is based on a
classical weak formulation for \eqref{eq:navierstokes}:
Find $(\vecb{u},p) :(0,T]\rightarrow \vecb{V}\times Q:= \vecb{H}_0^1(\Omega) \times L_0^2(\Omega)$ such that
\begin{align}\label{eq:navierstokes_weak}
\begin{array}{rcll}
(\partial_t \vecb{u},\vecb{v})+((\vecb{u}\cdot\nabla)\vecb{u},\vecb{v})+(\nu \nabla \vecb{u},\nabla \vecb{v}) - (\mathrm{div} (\vecb{v}),p) & = & (\vecb{f},\vecb{v}) & \forall\ \vecb{v}\in \vecb{V},\\
(\mathrm{div} (\vecb{u}),q) & = & 0 &  \forall\ q\in Q,
\end{array}
\end{align}
and $\vecb{u}(0)=\vecb{u}^0$.
Here, $\vecb{H}_0^1(\Omega):=[H_0^1(\Omega)]^d$ with $H_0^1(\Omega)$ being the
Sobolev space of functions in $H^1(\Omega)$ with zero trace along $\partial \Omega$.
The space $ L_0^2(\Omega)$ collects all functions in $L^2(\Omega)$ with zero mean,
and $(\bullet,\bullet)$ denotes the usual $L^2$ inner product.

Developing physically consistent schemes for \eqref{eq:navierstokes} or
\eqref{eq:navierstokes_weak}, in the meanings described next, is a challenging topic,
since there are several invariant structural (physical) properties or balance laws
to be taken into account.
They comprise the pointwise conservation of mass
(i.e., the divergence-free property of the velocity), and
the balance laws of kinetic energy, linear momentum, angular momentum,
enstrophy, vorticity and helicity \cite{Rebholz2010,Evans2013b,Rebholz2017}.
Moreover, there is an invariance property that the velocity field is independent
from any gradient field force \cite{Linke2014on,LM2016,JLMNR:2017}.
All these properties are considered to be crucial in designing physically consistent
numerical schemes.
Simultaneously, from the mathematical point of view, one has to consider discrete inf-sup stability and
the continuity requirement when choosing discrete space pairs
$(\vecb{V}_h, Q_h)$. Combining the physical and mathematical requests
is a challenging endeavor, since it is well known that some aspects of the physical consistency
and the satisfaction of the discrete inf-sup condition have competing requirements.
For some remarkable explorations in this regard, we refer to the
divergence-free finite element methods \cite{Zhang:2005,CKS2007,Zhang:2011,Zhang:2011b,Evans2013b,GN:2014,GN2018,GS2018},
pressure-robust reconstruction methods \cite{Linke2014on,LMT2016,LM:2016,LLMS2017,ALM2018},
EMAC formulation \cite{Rebholz2017,Rebholz2020,EMAC2019}, and other structure-preserving
methods such as \cite{Abramov2003,Rebholz2007,Palha2017,ABN2021}.

Among the physical properties mentioned above, the preservation of mass is strongly
related to the other properties. On the one hand, it has been shown in
\cite{Evans2013b,Rebholz2017} that $\vecb{H}^1$-conforming divergence-free
methods preserve proper balance of kinetic energy, linear momentum and
angular momentum in some appropriate sense. Similar properties for
$\vecb{H}(\operatorname{div})$-conforming divergence-free discontinuous
Galerkin (DG) methods can be also found in \cite{Chen2021}, except that
there is some artificial dissipation if upwind fluxes are chosen, which
modifies the balance of energy. However, it is also well-known that the
upwind fluxes can have even much better performance than central fluxes
with respect to the convergence order \cite{Guzman2017,HAN2021113365}.
Also the  pressure-robustness property, which means that the velocity error
is independent of the pressure, is usually ensured by a divergence-free
method, if no consistency errors arise from the inner product
with the force term \cite{JLMNR:2017}. Finally, the
divergence-free property is also related to another important property
called $\mathrm{Re}$-semi-robustness or convection-robustness \cite{Schroeder_2018,GJN21},
which means that the constants (including the Gronwall constant)
in the error estimates of the kinetic energy do not depend on the inverse of
the viscosity explicitly. It was shown in \cite{Schroeder_2018,GJN21} that a
large class of divergence-free $\vecb{H}^1$-conforming and
$\vecb{H}(\operatorname{div})$-conforming methods is convection-robust
with convergence order $k$, where $k$ is the order of velocity space.

The starting point of this paper is the family of divergence-free elements
designed in \cite{li2021low,JLMR2022} for the incompressible Stokes problem.
This family is easy to implement, divergence-free and inf-sup stable
on general shape-regular simplicial meshes. The main idea
is to employ
$\vecb{H}(\operatorname{div})$-conforming Raviart--Thomas bubbles to enrich
the generally non-inf-sup stable Scott--Vogelius finite element pair
$\vecb{P}_k-P_{k-1}^{\mathrm{disc}}$. Here,
$\vecb{P}_k$ denotes the vector-valued space of continuous piecewise
polynomials of order $k$ and $P_{k-1}^{\mathrm{disc}}$ denotes the scalar
spaces of discontinuous piecewise polynomials of order $k-1$.

Thus, the enriched velocity space consists of a classical $\vecb{H}^1$-conforming
part and a (small) $\vecb{H}(\operatorname{div})$-conforming part.
It combines the advantages
of divergence-free $\vecb{H}^1$-conforming and
$\vecb{H}(\operatorname{div})$-conforming methods: compared to the pure
divergence-free $\vecb{H}^1$-conforming methods on general meshes
\cite{GN:2014,GN2018, SH2019}, the relaxation of the continuity
requirement
of the bubble part allows a much simpler construction and implementation; opposite to
pure $\vecb{H}(\operatorname{div})$-conforming schemes
\cite{CKS2007, Wang2007,Anaya_2019}, the new formulation
only consists of volume integrals, and for $k\geq d$ the scheme is
parameter-free for the Stokes problem. Moreover, similarly to the HDG schemes
\cite{Lehrenfeld2010,Lehrenfeld_2016,LLS2018},
the proposed methods can be reduced to a $\vecb{P}_k-P_0$ problem via
static condensation, for any $k\ge 1$,
in this way decreasing the dimension of the global problem notably.

The goal of the present paper consists in extending the methods from
\cite{li2021low,JLMR2022} from the steady-state Stokes equations to the transient Navier--Stokes equations,
with particular consideration of the above mentioned  physical properties.
It will be shown that the suggested schemes have similar properties as
pure divergence-free $\vecb{H}^1$-conforming or
$\vecb{H}(\operatorname{div})$-conforming methods in the sense
that they preserve linear momentum and angular momentum, satisfy
pressure-robustness and they are convection-robust. In addition, they
maintain most of the particular
advantages mentioned before for the Stokes case.

The main difficulty for achieving all these favorable properties is the discretization of the
convection term, which requires a careful design. On the one hand, since
$\vecb{H}^1$-conforming elements are also
$\vecb{H}(\operatorname{div})$-conforming, the upwind
$\vecb{H}(\operatorname{div})$-conforming DG formulation is one possible
choice. On the other hand, if the convective term is treated implicitly, the
face integrals of the DG upwinding increase the coupling of degrees of freedom,
which compromises some of the original advantages and motivations in
\cite{li2021low,JLMR2022}.
Another possible choice is motivated by \cite{LR2021NS},
where a structure-preserving convective formulation was proposed for
the pressure-robust reconstruction schemes of \cite{LM:2016}.
A similar formulation is proposed here, which consists of
volume integrals only. We like to emphasize that there is a fundamental
difference between the method studied here and the one in \cite{LM:2016,LR2021NS}. The former
is indeed a nonconforming divergence-free method and the latter one is a
conforming, but non-divergence-free method with reconstruction.
In the present paper both choices for the convection term
are analyzed. The analysis of the convection form inspired by
\cite{LR2021NS} suggests to add one of two stabilizations for the
$\vecb{H}(\operatorname{div})$-conforming part to improve the error bounds.
One of them is proposed in \cite{LR2021NS}, and interestingly,
for $k\geq d$, the other one is
similar to a grad-div stabilization \cite{Olshanskii2002,Olshanskii2004,CELR2011,DGJN18}
despite the different purpose.

Finally, it should be stressed that compared to the method in
\cite{LR2021NS}, the novel methods possess some features that are of particular
interest in practice. They compute an exactly divergence-free velocity solution,
which means that the mass is conserved pointwise.
Moreover, the method in \cite{LR2021NS} uses classical
pressure-discontinuous Stokes elements whose bubbles are polynomials
of higher order, that require higher order quadrature rules especially
in three dimensions. The methods suggested in the present paper keep the polynomial
order $k$ for all ansatz functions in any dimensions.
Most importantly, these methods
are able to be reduced to a $\vecb{P}_k-P_0$ system due to the special
construction of the enrichment and divergence constraint, so that the dimension of the global
problem is reduced notably.

\medskip
The remainder of the paper is organized as follows. Section~\ref{sec:preliminaries}
introduces the notation and describes the involved finite element spaces.
Section~\ref{sec:discretization}
presents and discusses two discrete formulations of the Navier--Stokes problem
and analyzes conservation or balance properties for
kinetic energy, linear momentum, and angular momentum.
A pressure-robust and convection-robust error estimate for the
time-continuous discrete schemes is given in Section~\ref{sec:erroranalysis}.
Section~\ref{sec:reduced_scheme} discusses the possibility to reduce
the scheme to a $\vecb{P}_k-P_0$ system. Section~\ref{sec:num} reports on some numerical
studies that verify the convection-robust convergence order
and illustrate the overall performance of the proposed methods
in some benchmark problems, such as the
classical Kelvin--Helmholtz instability. Section~\ref{sec:conclusions}
draws some conclusions and gives an outlook on aspects that deserve
further attention in the future.

\section{Preliminaries}
\label{sec:preliminaries}
This section introduces the notation and recalls the main ideas
of the proposed Raviart--Thomas enrichment spaces
for the Scott--Vogelius finite element pairs from \cite{li2021low,JLMR2022}.

\subsection{Notation}
\label{sec:notation}
Consider a regular triangulation into simplices
\(\mathcal{T}\) of the domain \(\Omega\)
with nodes \(\mathcal{N}\) and facets \(\mathcal{F}\).
The subset $\mathcal{F}^0$ denotes all interior faces.
The diameter of an element $T\in \mathcal{T}$
is denoted by $h_T$
and gives rise to the local mesh-width function
$h_\mathcal{T}$ via $h_\mathcal{T}|_T := h_T$ for all $T \in \mathcal{T}$.
The maximum mesh-width is given by $h:=\max_{T\in\mathcal{T}}h_T$.
The vector $\vecb{n}_T$ defines
the outer unit normal vector along the boundary $\partial T$ of a simplex
$T \in \mathcal{T}$.
On a face $F$, the notation $\jump{\bullet}$ denotes the jump of $\bullet$
and $\average{\bullet}$ denotes its average value.

On a subdomain $\omega$, the space of all scalar-valued polynomials of order
$k$ on $\omega$ is denoted by \(P_k(\omega)\) and is written in bold, i.e.,
\(\vecb{P}_k(\omega)\), in case of vector-valued polynomials.
Piecewise continuous and discontinuous
polynomial spaces with respect to the triangulation are given by
\begin{align*}
{P}_k(\mathcal{T}) &:=\left\{q_h\in H^1(\Omega): q_h|_T\in P_k(T) \text{ for all } T\in \mathcal{T}\right\},\\
{{P}}_k^{\mathrm{disc}}(\mathcal{T}) &:=\left\{q_h\in L^2(\Omega): q_h|_T\in P_k(T) \text{ for all } T\in \mathcal{T}\right\}.
\end{align*}

The suggested enrichment relies on specially chosen Raviart--Thomas function.
The space of all Raviart--Thomas functions of order $k$ on a cell $T \in \mathcal{T}$
is given by
\begin{align*}
  \vecb{RT}_{k}(T) := \left\lbrace \vecb{v} \in \vecb{L}^2 (T) :\, \exists \vecb{p} \in \vecb{P}_k(T), q \in {P}_k(T), \ \vecb{v}|_T(\vecb{x}) = \vecb{p}(\vecb{x}) + q(\vecb{x}) \vecb{x} \right\rbrace.
\end{align*}
Their $H(\mathrm{div})$-conforming combinations
define the global space
\begin{align*}
  \vecb{RT}_{k}(\mathcal{T}) := \left\lbrace \vecb{v} \in \vecb{H}(\mathrm{div},\Omega) : \forall T \in \mathcal{T} \, \vecb{v}|_T \in \vecb{RT}_{k}(T) \right\rbrace.
\end{align*}
The subspace of interior Raviart--Thomas bubble functions reads
\begin{align*}
\vecb{RT}_k^{\mathrm{int}}(\mathcal{T}):=\left\lbrace\vecb{v}\in\vecb{RT}_k(\mathcal{T}): \vecb{v}\cdot\vecb{n}_T|_{\partial T}=0 \text{ for all } T\in\mathcal{T}\right\rbrace.
\end{align*}
This space can be further decomposed into
\begin{align*}
\vecb{RT}_k^{\mathrm{int}}(\mathcal{T})=\vecb{RT}_{k,0}^{\mathrm{int}}(\mathcal{T})\oplus \widetilde{\vecb{RT}}_k^{\mathrm{int}}(\mathcal{T}),
\end{align*}
where the first part consists of only divergence-free functions
and the second part $\widetilde{\vecb{RT}}_k^{\mathrm{int}}(\mathcal{T})$
is its arbitrary but fixed complement space.
Then, since the only divergence-free function in
$\widetilde{\vecb{RT}}_k^{\mathrm{int}}(\mathcal{T})$
is the zero function, the divergence operator on this space is injective
and allows the following estimate under a mild requirement in Remark~\ref{rem:tildert_unique}.
\begin{lemma}[\cite{JLMR2022}]\label{lem:bound_vhr}
  For any $\vecb{v}_h \in \widetilde{\vecb{RT}}_k^{\mathrm{int}}(\mathcal{T})$, there holds
  the inequality
      \begin{equation}\label{ieq:L2_L2div}
      \left\|\vecb{v}_h\right\|_{\vecb{L}^2(T)}\leq C h_{T} \left\|\mathrm{div}\left(\vecb{v}_h\right)\right\|_{L^2(T)}\quad \text{for all } T\in \mathcal{T}.
      \end{equation}
  \end{lemma}

\begin{remark}
\label{rem:tildert_unique}
In general, the space $\widetilde{\vecb{RT}}_k^{\mathrm{int}}(\mathcal{T})$
is not unique for $k>1$. We only require that it breaks into local spaces
$\widetilde{\vecb{RT}}_k^{\mathrm{int}}(T), T\in \mathcal{T}$,
that have the same structure in the sense that all of them are connected
to the same reference space via Piola's transformation
(see e.g.\ \cite[Eq.~2.1.69]{BBF:book:2013}).
Since these spaces are characterized by the normal trace and the divergence,
which are preserved (in a scaled meaning) by Piola's transformation,
this requirement is natural.
\end{remark}

Furthermore, the subspace of elementwise zero-mean functions in $P_{k}^{\mathrm{disc}}(\mathcal{T})$
reads
\begin{align}\label{eq:def_tildepk_disc}
\widetilde{P}_{k}^{\mathrm{disc}}(\mathcal{T}):=\left\{q_h\in {P}_{k}^{\mathrm{disc}}(\mathcal{T}): (q_h, 1)_T=0 \text{ for all } T\in\mathcal{T}\right\}.
\end{align}
For $k=0$, one obtains $\vecb{RT}_0^{\mathrm{int}}(\mathcal{T})=\widetilde{\vecb{RT}}_0^{\mathrm{int}}(\mathcal{T})=\left\{\vecb{0}\right\}$ and $\widetilde{P}_{0}^{\mathrm{disc}}(\mathcal{T})=\lbrace 0 \rbrace$. It also holds $\vecb{RT}_1^{\mathrm{int}}(\mathcal{T})=\widetilde{\vecb{RT}}_1^{\mathrm{int}}(\mathcal{T})$ because there is no divergence-free interior bubble in $\vecb{RT}_1$.

\smallskip
Throughout this paper, for any (scalar or vector-valued) finite element space $\mathcal{S}$
(with or without an argument like $\mathcal{T}$), its local version on each element $T$
is denoted by $\mathcal{S}(T)$ if not specially indicated.
The symbol $\pi_\mathcal{S}$ (or $\pi_{\mathcal{S}(T)}$)
denotes the $L^2$ projection operator onto $\mathcal{S}$
(or $\mathcal{S}(T)$, respectively).
$\|\bullet\|_{W^{m,p}(\omega)}$
is used to denote the $W^{m,p}$ Sobolev norm of
$\bullet$ on the domain $\omega$. By convention, $\omega$ is
omitted if $\omega=\Omega$, and the $L^2$ norm of $\bullet$
is then simply denoted by $\|\bullet\|$.

\subsection{Raviart--Thomas enriched Scott--Vogelius finite element pair}
\label{sec:enrichment_principle}
For $k \geq 1$, consider
the $\vecb{H}^1$-conforming velocity ansatz
space of piecewise vector-valued polynomials
\begin{align*}
	\vecb{V}_h^{\mathrm{ct}} := \vecb{P}_k(\mathcal{T}) \cap \vecb{V}
\end{align*}
and the desired pressure space
\begin{align*}
  Q_h := {P}_{k-1}^{\mathrm{disc}}(\mathcal{T}) \cap Q.
\end{align*}
These are the ansatz spaces for the classical
Scott--Vogelius finite element method which are known to be not inf-sup stable
in general. The main idea of \cite{li2021low,JLMR2022} is to enrich
the velocity spaces by some specially chosen Raviart--Thomas functions collected in the space $\vecb{V}_h^\mathrm{R}$.
The characteristic property of $\vecb{V}_h^\mathrm{R}$ is that
\begin{align*}
\mathrm{div}(\vecb{V}_h^\mathrm{R})\oplus_{L^2} \widehat{Q}_h=Q_h
\end{align*}
such that $\vecb{V}_h^\mathrm{ct}\times \widehat{Q}_h$ is inf-sup stable for some subspace $\widehat{Q}_h\subseteq Q_h$. The choice of $\widehat{Q}_h$ is not unique in general. Suggested by \cite{JLMR2022}, well-known inf-sup stability results from literature \cite{BBF:book:2013,Joh16} allow to use $\widehat{Q}_h=P_{k-d}^\mathrm{disc}(\mathcal{T})\cap Q$ for $k\geq d$ and $\widehat{Q}_h=\{0\}$ for $k<d$.
The corresponding Raviart--Thomas enrichment space $\vecb{V}_h^\mathrm{R}$ reads
\begin{align}\label{eq:def_vhr}
\begin{aligned}
\vecb{V}_h^\mathrm{R}=\begin{cases}
\vecb{RT}_{0}(\mathcal{T}) \cap \vecb{H}_0(\mathrm{div},\Omega) \quad & k=1,\\
(\vecb{RT}_{0}(\mathcal{T}) \cap \vecb{H}_0(\mathrm{div},\Omega)) \oplus \vecb{RT}_{1}^{\text{int}}(\mathcal{T})\quad & k=2, d=3,\\
\left\{\vecb{v}_{h}\in\widetilde{\vecb{RT}}_{k-1}^{\text{int}}(\mathcal{T}): \mathrm{div}(\vecb{v}_h)\in \widehat{Q}_h^\perp\right\}\quad & k\geq d,
\end{cases}
\end{aligned}
\end{align}
where
\begin{align*}
\vecb{H}_0(\mathrm{div},\Omega)):=\left\{\vecb{v}\in \vecb{H}(\mathrm{div},\Omega)): \vecb{v}\cdot\vecb{n}=0 \text{ on } \partial\Omega \right\}.
\end{align*}
Some explicit basis functions for \eqref{eq:def_vhr} can be found in
\cite{JLMR2022}. Later, for the construction of the reduced scheme
in Section~\ref{sec:reduced_scheme}, for all $k\geq d$ we choose $\widehat{Q}_h=P_0^\mathrm{disc}(\mathcal{T})$, which leads to a larger enrichment space $\vecb{V}_h^\mathrm{R}$ for $k>d$ but allows for some procedure that results in a much smaller system.

Define $\vecb{V}_h:=\vecb{V}_h^\mathrm{ct}\times \vecb{V}_h^\mathrm{R}$ and
$\vecb{V}(h):=\vecb{V}\times \vecb{V}_h^\mathrm{R}$. Throughout the paper,
the superscripts `$\mathrm{ct}$', `$\mathrm{R}$', and `$\mathrm{s}$' are
employed for $\vecb{v}\in \vecb{V}(h)$ in the following way:
$\vecb{v}^\mathrm{ct}\in \vecb{V}$ and $\vecb{v}^\mathrm{R}\in \vecb{V}_h^\mathrm{R}$
denote the $\vecb{H}^1$-conforming component and $\vecb{H}(\mathrm{div})$-conforming
component of $\vecb{v} = (\vecb{v}^\mathrm{ct}, \vecb{v}^\mathrm{R})$, and
$\vecb{v}^\mathrm{s}\in \vecb{V}+ \vecb{V}_h^\mathrm{R}$ denotes the
summation of both components, i.e., $\vecb{v}^\mathrm{s}:=\vecb{v}^\mathrm{ct}+\vecb{v}^\mathrm{R}$
which lives in $\vecb{H}(\mathrm{div}, \Omega)$ only.
Moreover, any
$\vecb{v}^\mathrm{R}\in \vecb{V}_h^\mathrm{R}$ can be split into
\begin{align*}
\vecb{v}^\mathrm{R}=:\vecb{v}^\mathrm{RT_0}+\widetilde{\vecb{v}}^\mathrm{R}
=\sum_{F\in\mathcal{F}^0}\mathrm{dof}_F(\vecb{v}_h^{\mathrm{RT}_0})
\vecb{\psi}_F+\widetilde{\vecb{v}}_h^{\mathrm{R}}\in
\vecb{RT}_0(\mathcal{T})\oplus\widetilde{\vecb{RT}}_{k-1}^\mathrm{int}(\mathcal{T}),
\end{align*}
where $\vecb{\psi}_F, F\in\mathcal{F}^0$, are the basis functions
of $\vecb{RT}_0(\mathcal{T})\cap \vecb{H}_0(\mathrm{div},\Omega)$, and
$\mathrm{dof}_F: \vecb{RT}_0(\mathcal{T})\rightarrow \mathbb{R}$
represents the degree of freedom functionals corresponding to $\vecb{\psi}_F$.
Note, that for $k\geq d$, no lowest-order Raviart--Thomas functions are involved
and therefore it holds
$\vecb{v}^\mathrm{RT_0}=\vecb{0}$ for all $\vecb{v}^\mathrm{R}\in \vecb{V}_h^\mathrm{R}$.

The divergence-free subspaces of $\vecb{V}$, $\vecb{V}(h)$, and $\vecb{V}_h$ are defined as
\begin{align*}
\vecb{Z}    & :=\left\{\vecb{v}\in \vecb{V}: \mathrm{div}(\vecb{v})=0\right\},\\
\vecb{Z}(h) & :=\left\{\vecb{v}\in \vecb{V}(h): \mathrm{div}(\vecb{v}^\mathrm{s})=0\right\}, \text{ and}\\
\vecb{Z}_h  & :=\left\{\vecb{v}_h\in \vecb{V}_h: \mathrm{div}(\vecb{v}_h^\mathrm{s})=0\right\}, \text{ respectively.}
\end{align*}

\section{Discretization of the Navier--Stokes equations}
\label{sec:discretization}
This section discusses the extension of the enrichment strategy
from the previous section for the Stokes problem to the full Navier--Stokes
problem and some direct structural properties.

\subsection{Discretization of the linear parts}
The discrete counterpart of \eqref{eq:navierstokes_weak} applies the same bilinear forms
used to discretize the stationary Stokes problem in \cite{li2021low,JLMR2022}, i.e.,
\begin{align*}
  a_h(\vecb{u},\vecb{v})
  := (\nabla\vecb{u}^{\mathrm{ct}}, \nabla\vecb{v}^{\mathrm{ct}})
     - (\Delta_\text{pw} \vecb{u}^{\mathrm{ct}}, \vecb{v}^{\mathrm{R}})
     + (\Delta_\text{pw} \vecb{v}^{\mathrm{ct}}, \vecb{u}^{\mathrm{R}})
     + a_h^\mathrm{D}(\vecb{u}^{\mathrm{RT}_0}, \vecb{v}^{\mathrm{RT}_0}),
\end{align*}
where $\Delta_\text{pw}$ is the piecewise Laplacian operator, and
\begin{align*}
  b(\vecb{v},q) := -(\mathrm{div}(\vecb{v}^{\mathrm{s}}), q),
\end{align*}
and
\begin{align*}
a_h^\mathrm{D}(\vecb{u}^{\mathrm{RT}_0}, \vecb{v}^{\mathrm{RT}_0}):=\alpha \sum_{F\in\mathcal{F}^0}\mathrm{dof}_F(\vecb{u}_h^{\mathrm{RT}_0}) \mathrm{dof}_F(\vecb{v}_h^{\mathrm{RT}_0}) \, (\mathrm{div} \vecb{\psi}_F, \mathrm{div} \vecb{\psi}_F),
\end{align*}
with a positive parameter $\alpha$, which is a stabilization only
needed for $k < d$.
Note that the coefficient matrix related to $a_h^\mathrm{D}$ is diagonal.
Also recall, that $|||\bullet|||^2:=a_h(\bullet,\bullet)$ induces a seminorm $|||\bullet|||$ on $\vecb{V}(h)$.
For the time derivative, we employ
\begin{align*}
d_h(\vecb{u},\vecb{v}):=\left(\vecb{u}^\mathrm{s},\vecb{v}^\mathrm{s}\right).
\end{align*}

\subsection{Discretization of the nonlinear terms}
For the discretization of the nonlinear term, which is of central importance in
Navier--Stokes simulations, consider the following nonlinear forms
\begin{align*}
c(\widehat{\vecb{w}},\widehat{\vecb{u}},\widehat{\vecb{v}}):=&((\widehat{\vecb{w}}\cdot\nabla_h)\widehat{\vecb{u}},\widehat{\vecb{v}}) \quad \text{for all } \widehat{\vecb{w}},\widehat{\vecb{u}},\widehat{\vecb{v}}\in \vecb{L}^2(\Omega),\\
  c_h^\mathrm{vol}(\vecb{w},\vecb{u},\vecb{v}) := & c(\vecb{w}^\mathrm{s},\vecb{u}^\mathrm{ct},\vecb{v}^\mathrm{s})-
  c(\vecb{w}^\mathrm{s},\vecb{v}^\mathrm{ct},\vecb{u}^\mathrm{R})\\
                                                = & c(\vecb{w}^\mathrm{s},\vecb{u}^\mathrm{ct},\vecb{v}^\mathrm{ct})+
  c(\vecb{w}^\mathrm{s},\vecb{u}^\mathrm{ct},\vecb{v}^\mathrm{R})-
  c(\vecb{w}^\mathrm{s},\vecb{v}^\mathrm{ct},\vecb{u}^\mathrm{R}),\\
  c_h^\mathrm{R}(\vecb{w},\vecb{u},\vecb{v})   := & c(\vecb{w}^\mathrm{s},\vecb{u}^\mathrm{R},\vecb{v}^\mathrm{R})-
  \sum_{F\in\mathcal{F}^0}\int_F(\vecb{w}^\mathrm{s}\cdot \vecb{n})\jump{\vecb{u}^\mathrm{R}}\cdot\average{\vecb{v}^\mathrm{R}}\ds,\\
  c_h^\mathrm{uw}(\vecb{w},\vecb{u},\vecb{v})  := & \frac{1}{2}\sum_{F\in\mathcal{F}^0}\int_F\abs{\vecb{w}^\mathrm{s}\cdot \vecb{n}}\jump{\vecb{u}^\mathrm{R}}\cdot\jump{\vecb{v}^\mathrm{R}}\ds,\\
\hspace{-5cm}\text{and}\hspace{3cm}&\\
  c_h^\mathrm{dG}(\vecb{w},\vecb{u},\vecb{v})  := & c(\vecb{w}^\mathrm{s},\vecb{u}^\mathrm{s},\vecb{v}^\mathrm{s}) -
  \sum_{F\in\mathcal{F}^0}\int_F(\vecb{w}^\mathrm{s}\cdot \vecb{n})\jump{\vecb{u}^\mathrm{R}}\cdot\average{\vecb{v}^\mathrm{s}}\ds
\end{align*}
for all $\vecb{u},\vecb{v},\vecb{w}\in \vecb{V}(h)$, where $\nabla_h$ is the piecewise gradient operator, $c_h^\mathrm{vol}$ is inspired by \cite{LR2021NS}, and $c_h^\mathrm{R}$, $c_h^\mathrm{uw}$ (upwind stabilization) and $c_h^\mathrm{dG}$ are inspired
by $\vecb{H}(\mathrm{div})$-conforming discontinuous Galerkin methods \cite{Guzman2017,Schroeder_2018}.

\begin{remark}[Relationship between $c_h^\mathrm{vol}$, $c_h^\mathrm{R}$ and $c_h^\mathrm{dG}$]
  Note that by continuity, it holds $\jump{\vecb{u}^\mathrm{ct}}=0$.
By an integration by parts, one has
\begin{align*}
c_h^\mathrm{vol}(\vecb{w},\vecb{u},\vecb{v})
:=  c(\vecb{w}^\mathrm{s},\vecb{u}^\mathrm{ct},\vecb{v}^\mathrm{s})
+ c(\vecb{w}^\mathrm{s},\vecb{u}^\mathrm{R},\vecb{v}^\mathrm{ct})
- \!\! \sum_{F\in\mathcal{F}^0}\int_{F}(\vecb{w}^\mathrm{s}\cdot\vecb{n})\jump{\vecb{u}^\mathrm{R}}\average{\vecb{v}^\mathrm{ct}}\ds.
\end{align*}
Then it is not very hard to verify that
\begin{align}\label{eq:rela_dGandvol}
c_h^\mathrm{dG}(\vecb{w},\vecb{u},\vecb{v}) =  c_h^\mathrm{vol}(\vecb{w},\vecb{u},\vecb{v}) + c_h^\mathrm{R}(\vecb{w},\vecb{u},\vecb{v}).
\end{align}
In a sense, $c_h^\mathrm{vol}$ can be regarded as an incomplete discontinuous Galerkin $\vecb{H}(\mathrm{div})$-conforming formulation for nonlinear terms, which consists of volume integrals only.
\end{remark}

\subsection{Stabilizations}
For the enrichment part the following two stabilizations are considered
\begin{align*}
\mathcal{S}_1(\vecb{u},\vecb{v}): = (\partial_t \vecb{u}^\mathrm{R},\vecb{v}^\mathrm{R})
\quad \text{and} \quad
\mathcal{S}_2(\vecb{u},\vecb{v}): = (h_\mathcal{T}^{-1}\vecb{u}^\mathrm{R},\vecb{v}^\mathrm{R}).
\end{align*}
Note that by \eqref{ieq:L2_L2div} for $k\geq d$ the second stabilization is equivalent to
\begin{align*}
\mathcal{S}_2(\vecb{u},\vecb{v})
\approx (h_\mathcal{T} \mathrm{div} \vecb{u}^\mathrm{R}, \mathrm{div} \vecb{v}^\mathrm{R})
= (h_\mathcal{T} \mathrm{div} \vecb{u}^\mathrm{ct}, \mathrm{div} \vecb{v}^\mathrm{ct})
\end{align*}
which can be seen as a grad-div stabilization of the $\vecb{H}^1$-conforming part. However, this one here scales with $h$ and thus
is weaker than the usual one used for non-divergence-free methods like Taylor--Hood \cite{DGJN18}. One should note that the grad-div-like stabilization here plays a different role than the grad-div stabilization for the Taylor--Hood element: the former is introduced to stabilize the nonconforming part, while the latter is used to improve the mass conservation of the discrete velocity solution. Opposite to the Taylor--Hood method, the present scheme is always pressure-robust and divergence-free.
\begin{remark}[Connection between $\mathcal{S}_1$ and $\mathcal{S}_2$]
We use backward Euler time stepping as an example.
Denote by $\Delta t$ the length of the time steps.
Then a corresponding discretization for $\mathcal{S}_1$ looks like
\begin{align*}\mathcal{S}_1^\mathrm{d}(\vecb{u}_h,\vecb{v}_h):=\Delta t^\mathrm{-1}\left((\vecb{u}_h^{\mathrm{R},n+1},\vecb{v}_h)-(\vecb{u}_h^{\mathrm{R},n},\vecb{v}_h)\right).\end{align*} The $\vecb{u}_h^{\mathrm{R},n}$
part should be shifted to the right-hand side in the computation.
If we ignore the right-hand side part and suppose $\Delta t \approx h$,
one can see that there is some similarity between $\mathcal{S}_1$ and $\mathcal{S}_2$.
\end{remark}

\subsection{Discretization schemes}
This paper investigates two discretization variants for
\eqref{eq:navierstokes} or \eqref{eq:navierstokes_weak}.

\smallskip
The first one employs the DG upwind discretization for the nonlinear term: Find $(\vecb{u}_h,p_h): (0, T]\rightarrow \vecb{V}_h\times Q_h$ such that
\begin{equation}\label{eq:semischeme1}
\begin{aligned}
d_h(\partial_t\vecb{u}_{h},\vecb{v}_h)
+c_h^\mathrm{dG}(\vecb{u}_h,\vecb{u}_h,\vecb{v}_h)
+c_h^\mathrm{uw}(\vecb{u}_h,\vecb{u}_h,\vecb{v}_h)\\
+\nu a_h(\vecb{u}_h,\vecb{v}_h) + b(\vecb{v}_h,p_h)
& = (\vecb{f},\vecb{v}_h^\mathrm{s}) && \text{for all } \vecb{v}_h\in \vecb{V}_h,\\
b(\vecb{u}_h,q_h) & = 0 &&  \text{for all } q_h\in Q_h,
\end{aligned}
\end{equation}
and $\vecb{u}_h(0)=\vecb{u}_h^0$ with $\vecb{u}_h^0$ being some
suitable approximation of $(\vecb{u}^0,\vecb{0})$.

\smallskip
The second discretization employs $c_h^\mathrm{vol}$
plus one of the two stabilizations
$\mathcal{S} \in \lbrace \mathcal{S}_1, \mathcal{S}_2\rbrace$:
Find $(\vecb{u}_h,p_h): (0, T]\rightarrow \vecb{V}_h\times Q_h$ such that
\begin{equation}\label{eq:semischeme2}
  \begin{aligned}
  d_h(\partial_t\vecb{u}_{h},\vecb{v}_h)
  +c_h^\mathrm{vol}(\vecb{u}_h,\vecb{u}_h,\vecb{v}_h)
  +\gamma\mathcal{S}(\vecb{u}_h,\vecb{v}_h)\\
  +\nu a_h(\vecb{u}_h,\vecb{v}_h) + b(\vecb{v}_h,p_h)
  & = (\vecb{f},\vecb{v}_h^\mathrm{s}) && \text{for all } \vecb{v}_h\in \vecb{V}_h,\\
  b(\vecb{u}_h,q_h) & = 0 && \text{for all } q_h\in Q_h,
  \end{aligned}
\end{equation}
and also $\vecb{u}_h(0)=\vecb{u}_h^0$ with $\vecb{u}_h^0$ being some suitable approximation of $(\vecb{u}^0,\vecb{0})$.
Here, $\gamma$ denotes some parameter to scale the stabilization.

By removing the Lagrange multiplier and
seeking the solution directly in the space of divergence-free functions,
both systems \eqref{eq:semischeme1} and \eqref{eq:semischeme2}
seek $\vecb{u}_h: (0, T]\rightarrow \vecb{Z}_h$ with
$\vecb{u}_h(0)=\vecb{u}_h^0$ such that
\begin{equation}\label{eq:semischeme1_divfree}
d_h(\partial_t\vecb{u}_{h},\vecb{v}_h)+c_h^\mathrm{dG}(\vecb{u}_h,\vecb{u}_h,\vecb{v}_h)+c_h^\mathrm{uw}(\vecb{u}_h,\vecb{u}_h,\vecb{v}_h)+\nu a_h(\vecb{u}_h,\vecb{v}_h)  =  (\vecb{f},\vecb{v}_h^\mathrm{s})
\end{equation}
and
\begin{align}\label{eq:semischeme2_divfree}
d_h(\partial_t\vecb{u}_{h},\vecb{v}_h)+c_h^\mathrm{vol}(\vecb{u}_h,\vecb{u}_h,\vecb{v}_h)+\gamma\mathcal{S}(\vecb{u}_h,\vecb{v}_h)+\nu a_h(\vecb{u}_h,\vecb{v}_h)  =  (\vecb{f},\vecb{v}_h^\mathrm{s})
\end{align}
for all $\vecb{v}_h\in\vecb{Z}_h$, respectively.

\subsection{EMA-conservation}
It has been shown in the paper \cite[Theorem 4]{Chen2021} that the
upwind DG formulation ($c_h^\mathrm{dG}+c_h^\mathrm{uw}$) is
momentum-conserving and angular momentum-conserving under the
assumption that all data is compactly supported
(see the assumption in Lemma~\ref{lem:MAconserving} below).
There it is also proven that the scheme with upwind DG formulation
is energy-stable. Although their analysis is based on the DG
formulation for the diffusion term, there is no essential
difference on EMA-conservation for our method. Thus this
subsection only discusses this aspect for the other scheme
\eqref{eq:semischeme2}.
\begin{lemma}\label{lem:skew-sym}
For any $(\vecb{u},\vecb{v},\vecb{w})\in \vecb{V}(h)\times \vecb{V}(h)\times \vecb{Z}(h)$, the trilinear form $c_h$ fulfills
\begin{align}\label{eq:skew-sym1}
c_h^\mathrm{vol}(\vecb{w},\vecb{u},\vecb{v})=-c_h^\mathrm{vol}(\vecb{w},\vecb{v},\vecb{u}).
\end{align}
\end{lemma}
\begin{proof}
The above identity follows immediately from $((\vecb{w}^\mathrm{s}\cdot\nabla)\vecb{u}^\mathrm{ct},\vecb{v}^\mathrm{ct})=-((\vecb{w}^\mathrm{s}\cdot\nabla)\vecb{v}^\mathrm{ct},\vecb{u}^\mathrm{ct})$ and the definition of $c_h^\mathrm{vol}$.
\end{proof}
Lemma~\ref{lem:skew-sym} implies that
\begin{align}\label{eq:skew-sym2}
c_h^\mathrm{vol}(\vecb{w},\vecb{v},\vecb{v})=0\quad\text{for all } \vecb{v}\in \vecb{V}(h),\vecb{w}\in\vecb{Z}(h).
\end{align}

The quantities under consideration
are the kinetic energy $E$,
linear momentum $M$
and angular momentum $M_{\boldsymbol{x}}$
defined by
\begin{align*}
  & E: \vecb{Z}(h)\rightarrow \mathbb{R}, \qquad & E(\boldsymbol{u})&:=\frac{1}{2} d_h(\boldsymbol{u},\boldsymbol{u})=\frac{1}{2} \int_{\Omega}\lvert\boldsymbol{u}^\mathrm{s}\rvert^{2} {~d} \boldsymbol{x},\\
  & M: \vecb{Z}(h)\rightarrow \mathbb{R}^{d}, \qquad & M(\boldsymbol{u})&:=\int_{\Omega} \boldsymbol{u}^\mathrm{s} ~{d} \boldsymbol{x},\\
  & M_{\boldsymbol{x}}: \vecb{Z}(h)\rightarrow \mathbb{R}^{3}, \qquad & M_{\boldsymbol{x}}(\boldsymbol{u})& :=\int_{\Omega} \boldsymbol{u}^\mathrm{s} \times \boldsymbol{x} ~{d} \boldsymbol{x},
\end{align*}
for any $\boldsymbol{u}\in \vecb{Z}(h)$.

\begin{lemma}\label{lem:Econserving}
Let $\vecb{u}_h$ be the solution of \eqref{eq:semischeme2}. It holds
\begin{align*}
E(\vecb{u}_h)+\mathcal{S}(\vecb{u}_h,\vecb{u}_h)+\nu|||\vecb{u}_h|||=(\vecb{f},\vecb{u}_h).
\end{align*}
\end{lemma}
\begin{proof}
  That is a direct consequence of testing \eqref{eq:semischeme2} with
  $\vecb{v}_h = \vecb{u}_h$ and the skew-symmetry of $c_h^\mathrm{vol}$.
\end{proof}

\begin{lemma}\label{lem:MAconserving}
Let $\vecb{u}_{h}$ be the solution of \eqref{eq:semischeme2}.
Assume that $\vecb{u}_h$, $p_h$ and $\vecb{f}$ are compactly
supported on a subdomain $\Omega_\omega$ such that there
exists an operator
$\chi: \vecb{L}^2(\Omega)\rightarrow \vecb{V}(h)$
satisfying $\chi(\vecb{g})|_{\Omega_\omega}=\vecb{g}$ for
$\vecb{g}=\vecb{e}_i, \vecb{x}\times \vecb{e}_i, i=1,\ldots, d$,
with $\vecb{e}_i\in \mathbb{R}^d$ being the unit vector with
respect to the $i$-th component. Then, the following identities
are satisfied:
\begin{displaymath}
\frac{d}{dt}M(\boldsymbol{u}_{h})=\int_{\Omega}\boldsymbol{f}~d\boldsymbol{x}, \quad \frac{d}{dt}M_{\boldsymbol{x}}(\boldsymbol{u}_{h})=\int_{\Omega}\boldsymbol{f}\times \boldsymbol{x}~d\boldsymbol{x}.
\end{displaymath}
\end{lemma}
\begin{proof}
The proof of this lemma is very similar to the proof of \cite[Theorem 2.2]{LR2021NS}.
\end{proof}

Lemma~\ref{lem:Econserving} shows that the energy is conserved
in some discrete sense that takes into account also the stabilization.
Lemma~\ref{lem:MAconserving} shows that the linear momentum and
angular momentum are conserved exactly.

\section{Pressure robust and convection-robust error estimate}
\label{sec:erroranalysis}
The section investigates a priori error estimates and shows
that pressure-robustness and convection-robustness can be attained for both suggested
schemes.

\subsection{Approximation properties and stability of a projection operator}
In this subsection a projection operator is designed, which will be used in the error analysis of the proposed schemes.
For any $\vecb{v}\in\vecb{V}$ and $k\geq d$, we define the Stokes projection operator
\begin{align*}
  \widehat{\Pi}_h^\mathrm{St}: \vecb{Z}\rightarrow \widehat{\vecb{Z}}_h
  := \lbrace \vecb{v}_h \in \vecb{V}_h^\text{ct} : b(\vecb{v}_h, q_h) = 0 \text{ for all } q_h \in \widehat{Q}_h \rbrace
\end{align*}
with the inf-sup stable sub-pair
$\vecb{V}_h^\text{ct}\times \widehat{Q}_h\subseteq\vecb{V}_h^\text{ct} \times Q_h$
mentioned in Section~\ref{sec:enrichment_principle} by seeking
$\widehat{\Pi}_h^\mathrm{St}\vecb{v}\in \widehat{\vecb{Z}}_h $ such that
  \begin{align*}
  \left(\nabla\widehat{\Pi}_h^\mathrm{St}\vecb{v},\nabla\vecb{w}\right)=\left(\nabla\vecb{v},\nabla\vecb{w}\right) \quad\text{for all } \vecb{w}\in \widehat{\vecb{Z}}_h.
  \end{align*}
  According to standard Stokes theory such as \cite{JLMNR:2017}, it satisfies
  \begin{align}\label{ieq:stprojection2}
  \|\vecb{v}-\widehat{\Pi}_h^\mathrm{St}\vecb{v}\|+h\|\nabla(\vecb{v}-\widehat{\Pi}_h^\mathrm{St}\vecb{v})\|
  \lesssim h \inf_{\vecb{w}\in \vecb{V}_h^\mathrm{ct}} \|\nabla(\vecb{v}-\vecb{w})\|.
  \end{align}
   In case $k<d$, where $\widehat{Q}_h$ is chosen to be the zero space, $\widehat{\Pi}_h^\mathrm{St}: \vecb{Z}\rightarrow \vecb{V}_h^\mathrm{ct}$ is defined as a quasi-interpolation operator \cite[Section 4.8]{BrennerScott:2008}, which satisfies
  \begin{equation}\label{ieq:appro_pict}
   \quad \|\vecb{v}-\widehat{\Pi}_h^\mathrm{St} \vecb{v}\|_{\vecb{L}^{p}(T)} +h_T \|\nabla(\vecb{v}-\widehat{\Pi}_h^\mathrm{St}\vecb{v})\|_{\vecb{L}^{p}(T)} \lesssim h_{T}^r\left|\vecb{v}\right|_{\vecb{W}^{r,p}(\omega(T))}, \, p=2,\infty,
  \end{equation}
  for all $T\in \mathcal{T}, 1\leq r\leq k$, with $\omega(T)$ being a suitable neighborhood containing $T$.

  \begin{assumption}\label{asmpt:inftybound0}
  For any $\vecb{v}\in \vecb{W}^{1,\infty}(\Omega)\cap \vecb{Z}$, we assume that the following estimate holds:
  \begin{align}\label{ieq:stprojection1}
  \|\vecb{v}-\widehat{\Pi}_h^\mathrm{St}\vecb{v}\|_{\vecb{L}^{\infty}} + h \|\nabla\widehat{\Pi}_h^\mathrm{St}\vecb{v}\|_{\vecb{L}^{\infty}}
  \lesssim h \|\nabla\vecb{v}\|_{\vecb{L}^{\infty}}.
  \end{align}
  \end{assumption}
  \begin{remark}
  For $k<d$, \eqref{ieq:stprojection1} can be derived from
  \eqref{ieq:appro_pict} by choosing $r=1$. For $k\geq d$, the bound for $\|\nabla\widehat{\Pi}_h^\mathrm{St}\vecb{v}\|_{\vecb{L}^{\infty}}$ was shown in \cite{girault_max-norm_2015} in case that $\Omega$ is convex and $\mathcal{T}$ is quasi-uniform. A similar assumption can be also found in \cite{Schroeder_2018,LR2021NS}.
  \end{remark}

Next, define $\Pi_h^\mathrm{St}: \vecb{Z}\rightarrow \vecb{Z}_h$ as
\begin{align}\label{def:PihSt}
\Pi_h^\mathrm{St} \vecb{v} :=
\begin{cases}
\left(\widehat{\Pi}_h^\mathrm{St} \vecb{v}, \mathcal{R} \widehat{\vecb{v}} \right) & \text{for } k \geq d,\\
\left(\widehat{\Pi}_h^\mathrm{St} \vecb{v}, (\Pi^{\mathrm{RT}_0} + \mathcal{R}) \widehat{\vecb{v}} \right)& \text{for } k < d,\\
\end{cases}
\ \text{ with } \ \widehat{\vecb{v}} := \vecb{v} - \widehat{\Pi}_h^\mathrm{St} \vecb{v},
\end{align}
where $\Pi^{\mathrm{RT}_0}$ is the usual interpolation operator for $\vecb{RT}_0$, and $\mathcal{R}$ is the $\mathrm{div}$-$L^2$ projection onto the $\vecb{H}(\operatorname{div})$ cell bubble space part of $\vecb{V}_h^\mathrm{R}$
with respect to the $(\mathrm{div}\bullet, \mathrm{div}\bullet)$ inner product and therefore it satisfies
\begin{align}\label{ieq:estimateR}
\|\mathrm{div}(\mathcal{R} \widehat{\vecb{v}})\|_{L^2(T)}\leq \|\mathrm{div}(\widehat{\vecb{v}})\|_{L^2(T)}=\|\mathrm{div}(\widehat{\Pi}_h^\mathrm{St} \vecb{v})\|_{L^2(T)}.
\end{align}
Note, that this resembles the
structure of the Fortin interpolator in \cite[Lemma 4.1]{JLMR2022} and by design it holds
$\Pi_h^\mathrm{St} \vecb{v} \in \vecb{Z}_h$ for any $\vecb{v} \in \vecb{Z}$.

\begin{lemma}\label{asmpt:inftybound}
For $\vecb{v}\in \vecb{Z}$, $\Pi_h^\mathrm{St}$ satisfies
\begin{align*}
\|\vecb{v}-(\Pi_h^\mathrm{St}\vecb{v})^\mathrm{s}\|+h\|\nabla(\vecb{v}-(\Pi_h^\mathrm{St}\vecb{v})^\mathrm{ct})\|
\lesssim h \inf_{\vecb{w}^\mathrm{ct}\in \vecb{V}_h^\mathrm{ct}} \|\nabla(\vecb{v}-\vecb{w}^\mathrm{ct})\|
\end{align*}
for $k\geq d$, and for $k<d$,
\begin{align*}
\|\vecb{v}-(\Pi_h^\mathrm{St}\vecb{v})^\mathrm{s}\|+h\|\nabla(\vecb{v}-(\Pi_h^\mathrm{St}\vecb{v})^\mathrm{ct})\|
+h\|h_\mathcal{T}^{-1}(\Pi_h^\mathrm{St}\vecb{v})^\mathrm{R}\|
\lesssim h^r\left|\vecb{v}\right|_{\vecb{H}^{r}},
\end{align*}
with $1\leq r\leq k$.
Additionally, under Assumption~\ref{asmpt:inftybound0}, we further have
\begin{align}\label{ieq:stprojection11}
  \|\vecb{v}-(\Pi_h^\mathrm{St}\vecb{v})^\mathrm{s}\|_{\vecb{L}^{\infty}}
  \lesssim h \|\nabla\vecb{v}\|_{\vecb{L}^{\infty}},
\end{align}
and
\begin{align}\label{ieq:inftybound}
\|\nabla(\Pi_h^\mathrm{St}\vecb{v})^\mathrm{ct}\|_{\vecb{L}^{\infty}}+\|h_\mathcal{T}^{-1}(\Pi_h^\mathrm{St}\vecb{v})^\mathrm{R}\|_{\vecb{L}^{\infty}}\lesssim \|\nabla\vecb{v}\|_{\vecb{L}^{\infty}}.
\end{align}
\end{lemma}

\begin{proof}
  Due to \eqref{ieq:stprojection2} and \eqref{ieq:appro_pict}, it suffices to show the $\vecb{L}^2$-bound and $\vecb{L}^\infty$-bound for the $\vecb{V}_h^\mathrm{R}$-part
  of \eqref{def:PihSt}. Consider an arbitrary
  $\vecb{v} \in \vecb{Z}$ with sufficient regularity. We prove the $\vecb{L}^2$-bound first. For the $\vecb{RT}_0$ part, from the approximation property of $\Pi^{\mathrm{RT}_0}$ \cite[Proposition 2.5.1]{BBF:book:2013}
  \begin{align*}
  \left\|\vecb{w}-\Pi^{\mathrm{RT}_0} \vecb{w}\right\|_{\vecb{L}^2(T)}\lesssim h_T \left\|\nabla \vecb{w}\right\|_{\vecb{L}^2(T)}
  \end{align*}
  and the triangle inequality, we have
  \begin{align}\label{ieq:estimateRT0}
  \left\|\Pi^{\mathrm{RT}_0} \vecb{w}\right\|_{\vecb{L}^2(T)}\lesssim \left\|\vecb{w}\right\|_{\vecb{L}^2(T)}+h_T \left\|\nabla \vecb{w}\right\|_{\vecb{L}^2(T)}
  \end{align}
   for all  $\vecb{w}\in \vecb{H}^1(T)$, which, together with \eqref{ieq:appro_pict}, implies that
  \begin{align*}
  \left\|\Pi^{\mathrm{RT}_0} \widehat{\vecb{v}}\right\|_{\vecb{L}^2(T)}\lesssim \left\|\widehat{\vecb{v}}\right\|_{\vecb{L}^2(T)}+h_T \left\|\nabla \widehat{\vecb{v}}\right\|_{\vecb{L}^2(T)}\lesssim h_T^r\left|\vecb{v}\right|_{\vecb{H}^{r}(\omega(T))}.
  \end{align*}
  For the higher order Raviart--Thomas part, according to Lemma~\ref{lem:bound_vhr} and \eqref{ieq:estimateR} one has
  \begin{align*}
  h_T^{-1}\|\mathcal{R} \widehat{\vecb{v}}\|_{\vecb{L}^2(T)}\lesssim  \|\mathrm{div}(\mathcal{R}\widehat{\vecb{v}})\|_{L^2(T)}
  \leq \|\mathrm{div}(\widehat{\vecb{v}})\|_{L^2(T)}
  \leq \|\nabla(\widehat{\vecb{v}})\|_{\vecb{L}^2(T)}.
  \end{align*}
  Then it follows from summation over $T\in\mathcal{T}$, \eqref{ieq:stprojection2}, \eqref{ieq:appro_pict}, and the fact that $(\Pi_h^\mathrm{St}\vecb{v})^\mathrm{R}=\mathcal{R} \widehat{\vecb{v}}$ or $(\Pi_h^\mathrm{St}\vecb{v})^\mathrm{R}=(\Pi^{\mathrm{RT}_0}+\mathcal{R}) \widehat{\vecb{v}}$ that
  \begin{align*}
  \|(\Pi_h^\mathrm{St}\vecb{v})^\mathrm{R}\|+h\|h_\mathcal{T}^{-1}(\Pi_h^\mathrm{St}\vecb{v})^\mathrm{R}\|\lesssim
  \begin{cases}  h \inf_{\vecb{w}^\mathrm{ct}\in \vecb{V}_h^\mathrm{ct}} \|\nabla(\vecb{v}-\vecb{w}^\mathrm{ct})\| &\text{ for } k\geq d,\\
   h^r\left|\vecb{v}\right|_{\vecb{H}^{r}} &\text{ for } k<d. \end{cases}
  \end{align*}

  Let us consider the $\vecb{L}^\infty$-bound.
  For $k \geq d$, no $\vecb{RT}_0$ functions are involved, and inverse inequalities,
  Lemma~\ref{lem:bound_vhr}, and \eqref{ieq:stprojection1} yield
  \begin{align*}
      h_T^{-1} \| \left(\Pi_h^\mathrm{St} \vecb{v}\right)^\mathrm{R} \|_{\vecb{L}^\infty(T)}
      & = h_T^{-1} \| \mathcal{R} \widehat{\vecb{v}} \|_{\vecb{L}^\infty(T)}\\
      & \lesssim h_T^{-1-d/2} \| \mathcal{R} \widehat{\vecb{v}} \|_{\vecb{L}^2(T)}\\
      & \lesssim h_T^{-d/2} \| \mathrm{div} \left(\mathcal{R} \widehat{\vecb{v}}\right) \|_{L^2(T)}\\
      & = h_T^{-d/2} \| \mathrm{div} \left( \widehat{\Pi}_h^\mathrm{St} \vecb{v} \right) \|_{L^2(T)}\\
      & \lesssim \| \mathrm{div} \left( \widehat{\Pi}_h^\mathrm{St} \vecb{v} \right) \|_{L^\infty(T)} \\
      & \leq \|\nabla \widehat{\Pi}_h^\mathrm{St} \vecb{v} \|_{\vecb{L}^\infty(T)}
      \lesssim \|\nabla \vecb{v} \|_{\vecb{L}^\infty(T)},
  \end{align*}
  where we also use the inequality $\| \vecb{w} \|^2_{\vecb{L}^2(T)}\leq h_T^{d/2} \| \vecb{w} \|_{\vecb{L}^\infty(T)}$ derived from
  \begin{align*}
    \| \vecb{w} \|^2_{\vecb{L}^2(T)}
    \leq \| \vecb{w} \|_{\vecb{L}^1(T)} \| \vecb{w} \|_{\vecb{L}^\infty(T)}
    & \leq \| 1 \|_{\vecb{L}^2(T)} \| \vecb{w} \|_{\vecb{L}^2(T)} \| \vecb{w} \|_{\vecb{L}^\infty(T)}\\
    & \leq h_T^{d/2} \| \vecb{w} \|_{\vecb{L}^2(T)} \| \vecb{w} \|_{\vecb{L}^\infty(T)}
  \end{align*}
  for all $\vecb{w} \in \vecb{L}^\infty(T)$.
  For $k < d$, $\mathcal{R}\widehat{\vecb{w}}$ can be bounded in a very similar way. For $\vecb{RT}_0$ part,
  it holds from inverse inequalities, \eqref{ieq:estimateRT0}, and \eqref{ieq:appro_pict} that
  \begin{align*}
    h_T^{-1} \| \Pi^{\mathrm{RT}_0} \widehat{\vecb{v}} \|_{\vecb{L}^\infty(T)}
    & \lesssim h_T^{-1-d/2} \| \Pi^{\mathrm{RT}_0} \widehat{\vecb{v}} \|_{\vecb{L}^2(T)}\\
    & \lesssim h_T^{-d/2} (h_T^{-1}\|\widehat{\vecb{v}} \|_{\vecb{L}^2(T)}+\| \nabla \widehat{\vecb{v}} \|_{\vecb{L}^2(T)})\\
    & \lesssim h_T^{-d/2} \| \nabla \vecb{v} \|_{\vecb{L}^2(\omega(T))}\\
    & \lesssim \| \nabla \vecb{v} \|_{\vecb{L}^\infty(\omega(T))}.
  \end{align*}
  Since, this holds for all $T \in \mathcal{T}$, one arrives at
  \begin{align*}
    \| h_{\mathcal{T}}^{-1} \left(\Pi_h^\mathrm{St} \vecb{v}\right)^\mathrm{R} \|_{\vecb{L}^\infty}
    \lesssim \| \nabla\vecb{v} \|_{\vecb{L}^{\infty}}.
  \end{align*}
  This concludes the proof.
\end{proof}

In what follows we will assume $\vecb{u}(t)\in\vecb{W}^{1,\infty}(\Omega)$ for all $t\leq T$. Note, that this also guarantees $\vecb{u}(t)\in \vecb{C}^0(\widebar{\Omega})$
according to the Sobolev imbedding theorem \cite{GiraultRaviart:1986}.
\subsection{Analysis of the upwind scheme \eqref{eq:semischeme1}}
To shorten the notation, we define $c_h:=c_h^\mathrm{dG}+c_h^\mathrm{uw}$.
Let $\vecb{u}$ be the solution of \eqref{eq:navierstokes_weak} and
$\widetilde{\vecb{u}}:=(\vecb{u},\vecb{0})\in \vecb{Z}(h)$ be its representation
in $\vecb{Z}_h$.
Due to the exact incompressibility of the functions in $\vecb{Z}_h$, one can check that $\widetilde{\vecb{u}}$ fulfills the identity
\begin{align}\label{eq:exactsolution}
d_h(\partial_t\widetilde{\vecb{u}},\vecb{v}_h)+c_h(\widetilde{\vecb{u}},\widetilde{\vecb{u}},\vecb{v}_h)+\nu a_h(\widetilde{\vecb{u}},\vecb{v}_h)  =  (\vecb{f},\vecb{v}_h^\mathrm{s}) \quad\text{for all } \vecb{v}_h\in \vecb{Z}_h.
\end{align}

Consider the decomposition of the error
$\vecb{e}_h:=\widetilde{\vecb{u}}-\vecb{u}_h=\vecb{\eta}+\vecb{\phi}_h$
into
\begin{align}\label{eqn:error_decomposition}
  \vecb{\eta} := \widetilde{\vecb{u}}-\Pi_h^\mathrm{St}\vecb{u}
  \quad \text{and} \quad
  \vecb{\phi}_h := \Pi_h^\mathrm{St}\vecb{u}-\vecb{u}_h.
\end{align}
Subtracting \eqref{eq:semischeme1_divfree} from \eqref{eq:exactsolution}
yields, for all $\vecb{v}_h\in \vecb{Z}_h$,
\begin{equation}\label{eq:error_equation}
\begin{aligned}
d_h(\partial_t\vecb{\phi}_h,\vecb{v}_h)+\nu a_h(\vecb{\phi}_h,\vecb{v}_h)=
& -d_{h}(\vecb{\eta}_{t},\vecb{v}_{h})-\nu a_h(\vecb{\eta},\vecb{v}_h)\\
& -c_{h}(\widetilde{\vecb{u}},\widetilde{\vecb{u}},\vecb{v}_{h})+c_{h}(\vecb{u}_{h},\vecb{u}_{h},\vecb{v}_{h}).
\end{aligned}
\end{equation}
The error equation \eqref{eq:error_equation} is used to estimate the velocity error bound below.

Moreover, for any $\vecb{w}_h$, one can define the upwind semi-norm
\begin{align*}
  |\vecb{v}_h|_{\vecb{w}_h,\mathrm{uw}}:=c_h^\mathrm{uw}(\vecb{w}_h,\vecb{v}_h,\vecb{v}_h)
\end{align*}
and bound the error of the convection terms
by the following lemma.
\begin{lemma}[\cite{Schroeder_2018}]
  \label{lem:estimate_donv_difference_upwind}
Let $\vecb{u}$ be the solution of \eqref{eq:navierstokes_weak} and $\vecb{u}_h$ be the solution of \eqref{eq:semischeme1_divfree}. Assume that $\vecb{u}\in L^{2}((0,T]; \vecb{W}^{1,\infty}(\Omega))$. Under Assumption~\ref{asmpt:inftybound0} it holds that
\begin{equation}\label{ieq:estimateNL_upwind}
\begin{aligned}
\lvert c_{h}(\widetilde{\vecb{u}},\widetilde{\vecb{u}},\vecb{\phi}_{h})&-c_{h}(\vecb{u}_{h},\vecb{u}_{h},\vecb{\phi}_{h})\rvert\leq - |\vecb{\phi}_h|_{\vecb{u}_h,\mathrm{uw}}^2
\\ &+ C[\|\vecb{u}\|_{\vecb{L}^\infty}\|\nabla_h\vecb{\eta}\|^2+(1+h^{-2})\|\nabla\vecb{u}\|_{\vecb{L}^\infty}\|\vecb{\eta}^\mathrm{s}\|^2] +  C\|\vecb{u}\|_{\vecb{W}^{1,\infty}} \|\vecb{\phi}_h\|^2.
\end{aligned}
\end{equation}
\end{lemma}
\begin{proof}
This estimate follows from taking
$\varepsilon_1 = \varepsilon_2 = 1$ and
$\varepsilon_3 = \varepsilon_4 = h$ in \cite[Lemma 5.5]{Schroeder_2018}.
\end{proof}

\begin{theorem}\label{thm:velocityestimate}
Let $\vecb{u}$ be the solution of \eqref{eq:navierstokes_weak} and
$\vecb{u}_h$ be the solution of \eqref{eq:semischeme1_divfree}.
Assume that $\vecb{u}\in L^{2}((0,T]; \vecb{W}^{1,\infty}(\Omega))$,
$\vecb{u}_{h}^{0}=\Pi_{h}^\mathrm{St}\vecb{u}(0)$,
and Assumption~\ref{asmpt:inftybound0} holds. The following estimate
holds with $K(\boldsymbol{u},T):=1+C\|\vecb{u}\|_{L^{1}((0,T];\vecb{W}^{1,\infty}(\Omega))}$:
\begin{multline}\label{ieq:velocityestimate}
E(\vecb{e}_{h}(T))
+\int_{0}^{T} \left(\nu|||\vecb{e}_{h}|||^{2}+|\vecb{e}_h|_{\vecb{u}_h,\mathrm{uw}}^2\right) ~dt\\
\begin{aligned}
& \leq E(\vecb{\eta}(T))+\int_{0}^{T} \left(\nu|||\vecb{\eta}|||^{2}+|\vecb{\eta}|_{\vecb{u}_h,\mathrm{uw}}^2\right) ~dt
+e^{{K(\vecb{u},T)}}\int_{0}^{T}\bigg\{ T E(\partial_t\vecb{\eta})\\
& \qquad + C \nu |||\vecb{\eta}|||^2
+C[\|\vecb{u}\|_{\vecb{L}^\infty}\|\nabla_h\vecb{\eta}\|^2+(1+h^{-2})\|\nabla\vecb{u}\|_{\vecb{L}^\infty}\|\vecb{\eta}^\mathrm{s}\|^2]\bigg\} ~dt.
\end{aligned}
\end{multline}
\end{theorem}
\begin{proof}
Taking $\vecb{v}_h=\vecb{\phi}_h$ in \eqref{eq:error_equation} gives
\begin{align}\label{eq:error_equationPhih}
\frac{d}{dt}E(\vecb{\phi}_{h})+\nu |||\vecb{\phi}_h|||^2
=-d_{h}(\partial_t\vecb{\eta},\vecb{\phi}_{h})-\nu a_h(\vecb{\eta},\vecb{\phi}_h)-c_{h}(\widetilde{\vecb{u}},\widetilde{\vecb{u}},\vecb{\phi}_{h})+c_{h}(\vecb{u}_{h},\vecb{u}_{h},\vecb{\phi}_{h}).
\end{align}
Note that
\begin{align}\label{ieq:time_estimate}
\vert d_{h}(\partial_t\vecb{\eta},\vecb{\phi}_{h})\vert\leq T E\left(\partial_t\vecb{\eta}\right)+\frac{1}{T}E\left(\vecb{\phi}_{h}\right)
\end{align}
and
\begin{align}\label{ieq:diff_estimate}
\nu\vert a_{h}(\vecb{\eta},\vecb{\phi}_{h})\vert\leq C \nu |||\vecb{\eta}|||^2+\frac{\nu}{2}|||\vecb{\phi}_h|||^2.
\end{align}

A combination of \eqref{eq:error_equationPhih}, \eqref{ieq:time_estimate}, \eqref{ieq:diff_estimate}, and \eqref{ieq:estimateNL_upwind} gives
\begin{align*}
\frac{d}{dt}E(\vecb{\phi}_{h})&+\frac{\nu}{2}|||\vecb{\phi}_{h}|||^{2}  +|\vecb{\phi}_h|_{\vecb{w}_h,\mathrm{uw}}^2  \leq  T E(\partial_t\vecb{\eta}) + \frac{1}{T}E(\vecb{\phi}_{h})+ C \nu |||\vecb{\eta}|||^2
\\ & + C[\|\vecb{u}\|_{\vecb{L}^\infty}\|\nabla_h\vecb{\eta}\|^2+(1+h^{-2})\|\nabla\vecb{u}\|_{\vecb{L}^\infty}\|\vecb{\eta}^\mathrm{s}\|^2]
 +  C\|\vecb{u}\|_{\vecb{W}^{1,\infty}} \|\vecb{\phi}_h\|^2.
\end{align*}
The Gronwall inequality, integration over $(0,T]$, and choosing
$\vecb{u}_h^0=\Pi_h^\mathrm{St}\vecb{u}(0)$ lead to
\begin{multline*}
E ( \vecb{\phi}_{h}(T))  +\int_{0}^{T} \left(\frac{\nu}{2}|||\vecb{\phi}_{h}|||^{2}+|\vecb{\phi}_h|_{\vecb{w}_h,\mathrm{uw}}^2\right) ~dt
 \leq e^{{K(\vecb{u},T)}}\int_{0}^{T}\bigg\{ T E(\partial_t\vecb{\eta})\\
+C \nu |||\vecb{\eta}|||^2+C[\|\vecb{u}\|_{\vecb{L}^\infty}\|\nabla_h\vecb{\eta}\|^2+(1+h^{-2})\|\nabla\vecb{u}\|_{\vecb{L}^\infty}\|\vecb{\eta}^\mathrm{s}\|^2]\bigg\} ~dt.
\end{multline*}
Then \eqref{ieq:velocityestimate} follows immediately.
This completes the proof.
\end{proof}
\subsection{Analysis of the scheme \eqref{eq:semischeme2}}
For the scheme \eqref{eq:semischeme2},
the error decomposition \eqref{eqn:error_decomposition}
this time leads to the error equation
\begin{equation}\label{eq:error_equation2}
\begin{aligned}
d_h(\partial_t\vecb{\phi}_h,\vecb{v}_h)&+\nu a_h(\vecb{\phi}_h,\vecb{v}_h)+\gamma\mathcal{S}(\vecb{\phi}_h,\vecb{v}_h)=-d_{h}(\vecb{\eta}_{t},\vecb{v}_{h})-\nu a_h(\vecb{\eta},\vecb{v}_h)\\
&-\gamma\mathcal{S}(\vecb{\eta},\vecb{v}_h)-c_{h}^\mathrm{vol}(\widetilde{\vecb{u}},\widetilde{\vecb{u}},\vecb{v}_{h})+c_{h}^\mathrm{vol}(\vecb{u}_{h},\vecb{u}_{h},\vecb{v}_{h})\text{ for all } \vecb{v}_h\in \vecb{Z}_h.
\end{aligned}
\end{equation}
Taking $\vecb{v}_h=\vecb{\phi}_h$ in \eqref{eq:error_equation2}
one obtains
\begin{equation}\label{eq:error_equation2phih}
\begin{aligned}
\frac{d}{dt}E(\vecb{\phi}_h)&+\nu |||\vecb{\phi}_h|||+\gamma\mathcal{S}(\vecb{\phi}_h,\vecb{\phi}_h)=-d_{h}(\vecb{\eta}_{t},\vecb{\phi}_{h})-\nu a_h(\vecb{\eta},\vecb{\phi}_h)\\
&-\gamma\mathcal{S}(\vecb{\eta},\vecb{\phi}_h)-c_{h}^\mathrm{vol}(\widetilde{\vecb{u}},\widetilde{\vecb{u}},\vecb{\phi}_{h})+c_{h}^\mathrm{vol}(\vecb{u}_{h},\vecb{u}_{h},\vecb{\phi}_{h}).
\end{aligned}
\end{equation}
According to \eqref{eq:rela_dGandvol} and the fact that
$c_{h}^\mathrm{vol}(\widetilde{\vecb{u}},\widetilde{\vecb{u}},\vecb{v}_{h})=c_{h}^\mathrm{dG}(\widetilde{\vecb{u}},\widetilde{\vecb{u}},\vecb{v}_{h})$,
we split the difference of the nonlinear terms as
\begin{equation}\label{eq:splitofchvol}
\begin{aligned}
c_{h}^\mathrm{vol}(\widetilde{\vecb{u}},\widetilde{\vecb{u}},\vecb{v}_{h})-c_{h}^\mathrm{vol}(\vecb{u}_{h},\vecb{u}_{h},\vecb{v}_{h})
&= c_{h}^\mathrm{vol}(\widetilde{\vecb{u}},\widetilde{\vecb{u}},\vecb{v}_{h})-c_{h}^\mathrm{dG}(\vecb{u}_{h},\vecb{u}_{h},\vecb{v}_{h})+
c_{h}^\mathrm{R}(\vecb{u}_{h},\vecb{u}_{h},\vecb{v}_{h})\\
& = c_{h}^\mathrm{dG}(\widetilde{\vecb{u}},\widetilde{\vecb{u}},\vecb{v}_{h})-c_{h}^\mathrm{dG}(\vecb{u}_{h},\vecb{u}_{h},\vecb{v}_{h})+
c_{h}^\mathrm{R}(\vecb{u}_{h},\vecb{u}_{h},\vecb{v}_{h}).
\end{aligned}
\end{equation}

\begin{lemma}[\cite{Schroeder_2018}]\label{lem:estimateNL_dG}
Let $\vecb{u}$ be the solution of \eqref{eq:navierstokes_weak} and $\vecb{u}_h$ be the solution of \eqref{eq:semischeme2_divfree}. Assume that $\vecb{u}\in L^{2}((0,T]; \vecb{W}^{1,\infty}(\Omega))$. Under Assumption~\ref{asmpt:inftybound0} it holds that
\begin{multline}\label{ieq:estimateNL_dG}
\lvert c_{h}^\mathrm{dG}(\widetilde{\vecb{u}},\widetilde{\vecb{u}},\vecb{\phi}_{h})-c_{h}^\mathrm{dG}(\vecb{u}_{h},\vecb{u}_{h},\vecb{\phi}_{h})\rvert\leq
\\
C [\|\vecb{u}\|_{\vecb{L}^\infty}\|\nabla_h\vecb{\eta}^\mathrm{s}\|^2+(1+h^{-2})\|\nabla\vecb{u}\|_{\vecb{L}^\infty}\|\vecb{\eta}^\mathrm{s}\|^2] +  C\|\vecb{u}\|_{\vecb{W}^{1,\infty}} \|\vecb{\phi}_h^\mathrm{s}\|^2.
\end{multline}
\end{lemma}
\begin{proof}
  The proof is similar to that of Lemma~\ref{lem:estimate_donv_difference_upwind},
  but without the estimate of the upwind term in \cite[Lemma 5.5]{Schroeder_2018}.
\end{proof}
\begin{lemma}
With the same assumption as in Lemma \ref{lem:estimateNL_dG}, it holds
\begin{align}\label{ieq:estimateNL_chr}
\lvert c_h^\mathrm{R}(\vecb{u}_h,\vecb{u}_h,\vecb{\phi}_h)\rvert
\lesssim \|\vecb{u}\|_{\vecb{L}^\infty}\|h_\mathcal{T}^{-1}\vecb{\eta}^\mathrm{R}\|\|\vecb{\phi}_h^\mathrm{R}\|
+ \|\nabla \vecb{u}\|_{\vecb{L}^\infty} (\|\vecb{\eta}^\mathrm{s}\|+\|\vecb{\phi}_h^\mathrm{s}\|)\|\vecb{\phi}_h^\mathrm{R}\|.
\end{align}
\end{lemma}
\begin{proof}
Note that $\widetilde{\vecb{u}}^\mathrm{R}=\vecb{0}$ and $c_h^\mathrm{R}(\vecb{w}_h,\bullet,\bullet)$ is also skew-symmetric as long as $\mathrm{div}(\vecb{w}_h^\mathrm{s})=0$. We have
\begin{align*}
-c_h^\mathrm{R}(\vecb{u}_h,\vecb{u}_h,\vecb{\phi}_h)& =c_h^\mathrm{R}(\widetilde{\vecb{u}},\widetilde{\vecb{u}},\vecb{\phi}_h)-c_h^\mathrm{R}(\vecb{u}_h,\vecb{u}_h,\vecb{\phi}_h)
\\ & =c_h^\mathrm{R}(\widetilde{\vecb{u}},\vecb{\eta},\vecb{\phi}_h)
+c_h^\mathrm{R}(\widetilde{\vecb{u}},\Pi_h^\mathrm{St}\vecb{u},\vecb{\phi}_h)-c_h^\mathrm{R}(\vecb{u}_h,\vecb{u}_h,\vecb{\phi}_h)
\\ & =c_h^\mathrm{R}(\widetilde{\vecb{u}},\vecb{\eta},\vecb{\phi}_h)
+c_h^\mathrm{R}(\vecb{e}_h,\Pi_h^\mathrm{St}\vecb{u},\vecb{\phi}_h)+c_h^\mathrm{R}(\vecb{u}_h,\vecb{\phi}_h,\vecb{\phi}_h)
\\ & = c_h^\mathrm{R}(\widetilde{\vecb{u}},\vecb{\eta},\vecb{\phi}_h)
+c_h^\mathrm{R}(\vecb{\eta},\Pi_h^\mathrm{St}\vecb{u},\vecb{\phi}_h)+c_h^\mathrm{R}(\vecb{\phi}_h,\Pi_h^\mathrm{St}\vecb{u},\vecb{\phi}_h),
\end{align*}
where the final equality follows from the fact $c_h^\mathrm{R}(\vecb{u}_h,\vecb{\phi}_h,\vecb{\phi}_h)=0$.
By the H\"{o}lder inequality, trace inequality, inverse inequality, and \eqref{ieq:inftybound} it holds
\begin{align*}
|c_h^\mathrm{R}(\widetilde{\vecb{u}},\vecb{\eta},\vecb{\phi}_h)| & \lesssim \|\vecb{u}\|_{\vecb{L}^\infty} \|\nabla_h \vecb{\eta}^\mathrm{R}\|\|\vecb{\phi}_h^\mathrm{R}\| + \|\vecb{u}\|_{\vecb{L}^\infty}\sum_{F\in\mathcal{F}^0}h_F^{-\frac{1}{2}}\|[[\vecb{\eta}^\mathrm{R}]]\|_{F}h_F^{\frac{1}{2}}\|[[\vecb{\phi}_h^\mathrm{R}]]\|_F\\
& \lesssim \|\vecb{u}\|_{\vecb{L}^\infty}(\|h_\mathcal{T}^{-1}\vecb{\eta}^\mathrm{R}\|+\|\nabla_h \vecb{\eta}^\mathrm{R}\|)\|\vecb{\phi}_h^\mathrm{R}\|
  \lesssim \|\vecb{u}\|_{\vecb{L}^\infty}\|h_\mathcal{T}^{-1}\vecb{\eta}^\mathrm{R}\|\|\vecb{\phi}_h^\mathrm{R}\|,
\end{align*}
\begin{align*}
\vert c_h^\mathrm{R}(\vecb{\eta},\Pi_h^\mathrm{St}\vecb{u},\vecb{\phi}_h)\vert
&     \lesssim \|\nabla_h(\Pi_h^\mathrm{St}\vecb{u})^\mathrm{R}\|_{\vecb{L}^\infty} \|\vecb{\eta}^\mathrm{s}\|\|\vecb{\phi}_h^\mathrm{R}\|\lesssim \|h_\mathcal{T}^{-1}(\Pi_h^\mathrm{St}\vecb{u})^\mathrm{R}\|_{\vecb{L}^\infty} \|\vecb{\eta}^\mathrm{s}\|\|\vecb{\phi}_h^\mathrm{R}\|
\\ &  \lesssim \|\nabla \vecb{u}\|_{\vecb{L}^\infty} \|\vecb{\eta}^\mathrm{s}\|\|\vecb{\phi}_h^\mathrm{R}\|,
\end{align*}
and
\begin{align*}
|c_h^\mathrm{R}(\vecb{\phi}_h,\Pi_h^\mathrm{St}\vecb{u},\vecb{\phi}_h)|
& \lesssim \|\nabla_h(\Pi_h^\mathrm{St}\vecb{u})^\mathrm{R}\|_{\vecb{L}^\infty} \|\vecb{\phi}_h^\mathrm{s}\|\|\vecb{\phi}_h^\mathrm{R}\| + \|(\Pi_h^\mathrm{St}\vecb{u})^\mathrm{R}\|_{\vecb{L}^\infty}\sum_{F\in\mathcal{F}^0}\|\vecb{\phi}_h^\mathrm{s}\cdot\vecb{n}\|_{F}\|[[\vecb{\phi}_h^\mathrm{R}]]\|_F\\
& \lesssim \|h_\mathcal{T}^{-1}(\Pi_h^\mathrm{St}\vecb{u})^\mathrm{R}\|_{\vecb{L}^\infty} \Big(\|\vecb{\phi}_h^\mathrm{s}\|\|\vecb{\phi}_h^\mathrm{R}\| + \sum_{F\in\mathcal{F}^0}h_F\|\vecb{\phi}_h^\mathrm{s}\cdot\vecb{n}\|_{F}
\|[[\vecb{\phi}_h^\mathrm{R}]]\|_F\Big)\\
& \lesssim \|\nabla \vecb{u}\|_{\vecb{L}^\infty} \|\vecb{\phi}_h^\mathrm{s}\|\|\vecb{\phi}_h^\mathrm{R}\|.
\end{align*}
Then \eqref{ieq:estimateNL_chr} follows immediately. This completes the proof.
\end{proof}
The terms produced by $c_h^\mathrm{R}$ in Lemma~\ref{ieq:estimateNL_chr} can be
stabilized by one of the two stabilizations $\mathcal{S}_1$ or $\mathcal{S}_2$.
For $\mathcal{S}=\mathcal{S}_1$, one can employ the Gronwall inequality.
For $\mathcal{S}=\mathcal{S}_2$, techniques similar
to the ones from the analysis of grad-div stabilizations \cite{DGJN18}
are applicable. In preparation for that,
the following
lemma states an intermediate estimate for the full convection term.
\begin{lemma}\label{lem:estimateNL_chr12}
In view of the two stabilizations $\mathcal{S}_1$ and $\mathcal{S}_2$, the following two
estimates hold:
\begin{multline*}
\lvert c_h^\mathrm{R}(\vecb{u}_h,\vecb{u}_h,\vecb{\phi}_h)\rvert
\lesssim
\begin{cases}
\|\vecb{u}\|_{\vecb{L}^\infty}\|h_\mathcal{T}^{-1} \vecb{\eta}^\mathrm{R}\|^2
+\|\nabla \vecb{u}\|_{\vecb{L}^\infty} (\|\vecb{\eta}^\mathrm{s}\|^2+\|\vecb{\phi}_h^\mathrm{s}\|^2)+\|\vecb{u}\|_{\vecb{W}^{1,\infty}}\|\vecb{\phi}_h^\mathrm{R}\|^2
& \text{for } \mathcal{S}_1,\\
\frac{1}{\gamma}\|\vecb{u}\|_{\vecb{L}^\infty}^2\|h_\mathcal{T}^{-\frac{1}{2}} \vecb{\eta}^\mathrm{R}\|^2+\frac{1}{\gamma}\|h_\mathcal{T}^{\frac{1}{2}}\nabla \vecb{u}\|_{\vecb{L}^\infty}^2 \left(\|\vecb{\eta}^\mathrm{s}\|^2+\|\vecb{\phi}_h^\mathrm{s}\|^2\right)+\frac{\gamma}{4}\|h_\mathcal{T}^{-\frac{1}{2}}\vecb{\phi}_h^\mathrm{R}\|^2
& \text{for } \mathcal{S}_2.
\end{cases}
\end{multline*}
According to \eqref{eq:splitofchvol},
a combination with Lemma~\ref{lem:estimateNL_dG} immediately implies
\begin{equation}\label{ieq:estimateNL_chr1}
\begin{aligned}
\lvert c_{h}^\mathrm{vol}(\widetilde{\vecb{u}},&\widetilde{\vecb{u}},\vecb{\phi}_{h})
-c_{h}^\mathrm{vol}(\vecb{u}_{h},\vecb{u}_{h},\vecb{\phi}_{h})\rvert
 \lesssim \|\vecb{u}\|_{\vecb{L}^\infty}\left(\|\nabla_h \vecb{\eta}^\mathrm{s}\|^2+\|h_\mathcal{T}^{-1} \vecb{\eta}^\mathrm{R}\|^2\right)\\
& +(1+h^{-2})\|\nabla\vecb{u}\|_{\vecb{L}^\infty}\|\vecb{\eta}^\mathrm{s}\|^2
  +\|\vecb{u}\|_{\vecb{W}^{1,\infty}} \left(\|\vecb{\phi}_h^\mathrm{s}\|^2+\|\vecb{\phi}_h^\mathrm{R}\|^2\right),
\end{aligned}
\end{equation}
and
\begin{equation}\label{ieq:estimateNL_chr2}
  \begin{aligned}
&\lvert c_{h}^\mathrm{vol}(\widetilde{\vecb{u}},
\widetilde{\vecb{u}},\vecb{\phi}_{h})
 -c_{h}^\mathrm{vol}(\vecb{u}_{h},\vecb{u}_{h},\vecb{\phi}_{h}) \rvert
  \lesssim \|\vecb{u}\|_{\vecb{L}^\infty}\|\nabla_h \vecb{\eta}^\mathrm{s}\|^2\\
&+\frac{1}{\gamma}\|\vecb{u}\|_{\vecb{L}^\infty}^2\|h_\mathcal{T}^{-\frac{1}{2}}\vecb{\eta}^\mathrm{R}\|^2
+\frac{1}{\gamma}\|h_\mathcal{T}^{\frac{1}{2}}\nabla \vecb{u}\|_{\vecb{L}^\infty}^2 \|\vecb{\eta}^\mathrm{s}\|^2
 +(1+h^{-2})\|\nabla\vecb{u}\|_{\vecb{L}^\infty}\|\vecb{\eta}^\mathrm{s}\|^2\\
&+\left(\|\vecb{u}\|_{\vecb{W}^{1,\infty}}+\frac{1}{\gamma}\|h_\mathcal{T}^{\frac{1}{2}}\nabla \vecb{u}\|_{\vecb{L}^\infty}^2\right) \|\vecb{\phi}_h^\mathrm{s}\|^2+\frac{\gamma}{4} \|h_\mathcal{T}^{-\frac{1}{2}}\vecb{\phi}_h^\mathrm{R}\|^2.
  \end{aligned}
\end{equation}
\end{lemma}
\begin{proof}
  Estimates \eqref{ieq:estimateNL_chr1} and \eqref{ieq:estimateNL_chr2} follow from
  \eqref{ieq:estimateNL_chr} and some differently weighted Young inequalities. The other estimates
  follow with Lemma~\ref{lem:estimateNL_dG} and a triangle inequality.
\end{proof}

\begin{theorem}\label{thm:velocityestimate2}
Let $\vecb{u}$ be the solution of \eqref{eq:navierstokes_weak} and $\vecb{u}_h$ be the solution of \eqref{eq:semischeme2_divfree}. Assume that $\vecb{u}\in L^{2}((0,T]; \vecb{W}^{1,\infty}(\Omega))$ and $\vecb{u}_{h}^{0}=\Pi_{h}^\mathrm{St}\vecb{u}(0)$. Under Assumption~\ref{asmpt:inftybound0} it holds for $\mathcal{S}=\mathcal{S}_1$:
\begin{equation}\label{ieq:velocityestimate21}
\begin{aligned}
& E(\vecb{e}_{h}(T))+\frac{\gamma}{2}\|\vecb{e}_h^\mathrm{R}(T)\|^2+\int_{0}^{T}\nu|||\vecb{e}_{h}|||^{2} ~dt \leq
E(\vecb{\eta}(T))+\frac{\gamma}{2}\|\vecb{\eta}^\mathrm{R}(T)\|^2\\
& \qquad +\int_{0}^{T}\nu|||\vecb{\eta}|||^{2} ~dt+e^{{K_1(\vecb{u},T)}}\int_{0}^{T}\bigg\{ T E(\partial_t\vecb{\eta})+C \nu |||\vecb{\eta}|||^2 + \frac{\gamma T}{2}\|\partial_t\vecb{\eta}^\mathrm{R}\|^2\\
& \qquad +C \|\vecb{u}\|_{\vecb{L}^\infty}\left(\|\nabla_h \vecb{\eta}^\mathrm{s}\|^2+\|h_\mathcal{T}^{-1} \vecb{\eta}^\mathrm{R}\|^2\right)
   + C (1+h^{-2})\|\nabla\vecb{u}\|_{\vecb{L}^\infty}\|\vecb{\eta}^\mathrm{s}\|^2\bigg\} ~dt,
\end{aligned}
\end{equation}
and for $\mathcal{S}=\mathcal{S}_2$:
\begin{equation}\label{ieq:velocityestimate22}
\begin{aligned}
& E(\vecb{e}_{h}(T))+\int_{0}^{T}\nu|||\vecb{e}_{h}|||^{2} + \frac{\gamma}{2}\|h_\mathcal{T}^{-\frac{1}{2}}\vecb{e}_h^\mathrm{R}\|^2 ~dt \leq
E(\vecb{\eta}(T))+\int_{0}^{T}\nu|||\vecb{\eta}|||^{2}\,dt\\
& \qquad +\int_{0}^{T}\frac{\gamma}{2}\|h_\mathcal{T}^{-\frac{1}{2}}\vecb{\eta}^\mathrm{R}\|^2 ~dt+e^{{K_2(\vecb{u},T)}}\int_{0}^{T}\bigg\{ T E(\partial_t\vecb{\eta})+C \nu |||\vecb{\eta}|||^2 + \frac{\gamma}{2}\|h_\mathcal{T}^{-\frac{1}{2}}\vecb{\eta}^\mathrm{R}\|^2\\
& \qquad +C \Big[\|\vecb{u}\|_{\vecb{L}^\infty}\|\nabla_h \vecb{\eta}^\mathrm{s}\|^2
  +\frac{1}{\gamma}\|\vecb{u}\|_{\vecb{L}^\infty}^2\|h_\mathcal{T}^{-\frac{1}{2}}\vecb{\eta}^\mathrm{R}\|^2+\frac{1}{\gamma}\|h_\mathcal{T}^{\frac{1}{2}}\nabla \vecb{u}\|_{\vecb{L}^\infty}^2 \|\vecb{\eta}^\mathrm{s}\|^2\\
& \qquad +(1+h^{-2})\|\nabla\vecb{u}\|_{\vecb{L}^\infty}\|\vecb{\eta}^\mathrm{s}\|^2\Big]\bigg\}\,dt,
\end{aligned}
\end{equation}
where
\begin{align*}
K_1(\boldsymbol{u},T)&:=1+\gamma/2+C(1+\frac{1}{\gamma})\|\vecb{u}\|_{L^{1}((0,T];\vecb{W}^{1,\infty}(\Omega))},\\
K_2(\boldsymbol{u},T)&:=1+C\Big(\|\vecb{u}\|_{L^{1}((0,T];\vecb{W}^{1,\infty}(\Omega))}+\frac{1}{\gamma}\|h_\mathcal{T}^{\frac{1}{2}}\nabla \vecb{u}\|_{L^{2}((0,T];\vecb{L}^{\infty}(\Omega))}\Big).
\end{align*}
\end{theorem}
\begin{proof}
Direct calculations imply
\begin{align}
&\mathcal{S}_1(\vecb{\phi}_h,\vecb{\phi}_h)=\frac{1}{2}\frac{d}{dt}\|\vecb{\phi}_h^\mathrm{R}\|^2,         &&   \vert \mathcal{S}_1(\vecb{\eta},\vecb{\phi}_h)\vert\leq
\frac{T}{2}\|\partial_t\vecb{\eta}^\mathrm{R}\|^2+\frac{1}{2T}\|\vecb{\phi}_h^\mathrm{R}\|^2, \label {ieq:estimate_S1}\\
&\mathcal{S}_2(\vecb{\phi}_h,\vecb{\phi}_h)= \|h_\mathcal{T}^{-\frac{1}{2}}\vecb{\phi}_h^\mathrm{R}\|^2,   &&   \vert \mathcal{S}_2(\vecb{\eta},\vecb{\phi}_h) \vert \leq
\frac{1}{2}\|h_\mathcal{T}^{-\frac{1}{2}}\vecb{\eta}^\mathrm{R}\|^2+\frac{1}{2}\|h_\mathcal{T}^{-\frac{1}{2}}\vecb{\phi}_h^\mathrm{R}\|^2. \label {ieq:estimate_S2}
\end{align}
For $i=1$ (i.e., $\mathcal{S}=\mathcal{S}_1$), a combination of \eqref{eq:error_equation2phih}, \eqref{ieq:time_estimate}, \eqref{ieq:diff_estimate}, \eqref{ieq:estimate_S1}, and \eqref{ieq:estimateNL_chr1} gives
\begin{multline*}
\frac{d}{dt}E(\vecb{\phi}_{h})+\frac{\gamma}{2}\frac{d}{dt}\|\vecb{\phi}_h^\mathrm{R}\|^2
   +\frac{\nu}{2}|||\vecb{\phi}_{h}|||^{2}
    \leq  T E(\partial_t\vecb{\eta}) + \frac{1}{T}E(\vecb{\phi}_{h})+ C \nu |||\vecb{\eta}|||^2\\
\begin{aligned}
& \qquad + \frac{\gamma T}{2}\|\partial_t\vecb{\eta}^\mathrm{R}\|^2+\frac{\gamma}{2T}\|\vecb{\phi}_h^\mathrm{R}\|^2 +  C \|\vecb{u}\|_{\vecb{L}^\infty}\left(\|\nabla_h \vecb{\eta}^\mathrm{s}\|^2+\|h_\mathcal{T}^{-1} \vecb{\eta}^\mathrm{R}\|^2\right)\\
& \qquad + C (1+h^{-2})\|\nabla\vecb{u}\|_{\vecb{L}^\infty}\|\vecb{\eta}^\mathrm{s}\|^2
  + C \|\vecb{u}\|_{\vecb{W}^{1,\infty}} \left(\|\vecb{\phi}_h^\mathrm{s}\|^2+\|\vecb{\phi}_h^\mathrm{R}\|^2\right).
\end{aligned}
\end{multline*}
Since $\|\vecb{\phi}_h^\mathrm{s}\|^2\leq 2E(\vecb{\phi}_h)$ and
$\|\vecb{\phi}_h^\mathrm{R}\|^2 = \frac{2}{\gamma} * \frac{\gamma}{2}\|\vecb{\phi}_h^\mathrm{R}\|^2$,
the Gronwall lemma yields
\begin{multline*}
E(\vecb{\phi}_{h}(T))
  +\frac{\gamma}{2}\|\vecb{\phi}_h^\mathrm{R}(T)\|^2
   +\int_0^T\frac{\nu}{2}|||\vecb{\phi}_{h}|||^{2}\,dt\\
\begin{aligned}
    & \leq  e^{{K_1(\vecb{u},T)}}\int_{0}^{T}\bigg\{ T E(\partial_t\vecb{\eta})+C \nu |||\vecb{\eta}|||^2
    + \frac{\gamma T}{2}\|\partial_t\vecb{\eta}^\mathrm{R}\|^2 \\
& \qquad +C \|\vecb{u}\|_{\vecb{L}^\infty}\left(\|\nabla_h \vecb{\eta}^\mathrm{s}\|^2+\|h_\mathcal{T}^{-1} \vecb{\eta}^\mathrm{R}\|^2\right)
   + C (1+h^{-2})\|\nabla\vecb{u}\|_{\vecb{L}^\infty}\|\vecb{\eta}^\mathrm{s}\|^2\bigg\} ~dt.
\end{aligned}
\end{multline*}
Then \eqref{ieq:velocityestimate21} follows.
For $i=2$ (i.e., $\mathcal{S}=\mathcal{S}_2$), a combination of
\eqref{eq:error_equation2phih}, \eqref{ieq:time_estimate},
\eqref{ieq:diff_estimate}, \eqref{ieq:estimate_S2}, and
\eqref{ieq:estimateNL_chr2} leads to
\begin{align*}
& \frac{d}{dt}E(\vecb{\phi}_{h})
  +\frac{\nu}{2}|||\vecb{\phi}_{h}|||^{2}+\frac{\gamma}{4}\|h_\mathcal{T}^{-\frac{1}{2}}\vecb{\phi}_h^\mathrm{R}\|^2
    \leq  T E(\partial_t\vecb{\eta}) + \frac{1}{T}E(\vecb{\phi}_{h})+ C \nu |||\vecb{\eta}|||^2
\\ &\ + \frac{\gamma}{2}\|h_\mathcal{T}^{-\frac{1}{2}}\vecb{\eta}^\mathrm{R}\|^2 +  C \Big[\|\vecb{u}\|_{\vecb{L}^\infty}\|\nabla_h \vecb{\eta}^\mathrm{s}\|^2
  +\frac{1}{\gamma}\|\vecb{u}\|_{\vecb{L}^\infty}^2\|h_\mathcal{T}^{-\frac{1}{2}}\vecb{\eta}^\mathrm{R}\|^2+\frac{1}{\gamma}\|h_\mathcal{T}^{\frac{1}{2}}\nabla \vecb{u}\|_{\vecb{L}^\infty}^2 \|\vecb{\eta}^\mathrm{s}\|^2\\
&\ +(1+h^{-2})\|\nabla\vecb{u}\|_{\vecb{L}^\infty}\|\vecb{\eta}^\mathrm{s}\|^2
  +\Big(\|\vecb{u}\|_{\vecb{W}^{1,\infty}}+\frac{1}{\gamma}\|h_\mathcal{T}^{\frac{1}{2}}\nabla \vecb{u}\|_{\vecb{L}^\infty}^2\Big) \|\vecb{\phi}_h^\mathrm{s}\|^2\Big].
\end{align*}
Also by Gronwall lemma we have
\begin{multline*}
E(\vecb{\phi}_{h}(T))
 +\int_0^T\frac{\nu}{2}|||\vecb{\phi}_{h}|||^{2}+\frac{\gamma}{4}\|h_\mathcal{T}^{-\frac{1}{2}}\vecb{\phi}_h^\mathrm{R}\|^2\,dt\\
\begin{aligned}
& \leq  e^{{K_2(\vecb{u},T)}}\int_{0}^{T}\bigg\{ T E(\partial_t\vecb{\eta})+C \nu |||\vecb{\eta}|||^2+ \frac{\gamma}{2}\|h_\mathcal{T}^{-\frac{1}{2}}\vecb{\eta}^\mathrm{R}\|^2\\
& \qquad +C \Big[\|\vecb{u}\|_{\vecb{L}^\infty}\|\nabla_h \vecb{\eta}^\mathrm{s}\|^2
  +\frac{1}{\gamma}\|\vecb{u}\|_{\vecb{L}^\infty}^2\|h_\mathcal{T}^{-\frac{1}{2}}\vecb{\eta}^\mathrm{R}\|^2+\frac{1}{\gamma}\|h_\mathcal{T}^{\frac{1}{2}}\nabla \vecb{u}\|_{\vecb{L}^\infty}^2 \|\vecb{\eta}^\mathrm{s}\|^2\\
& \qquad +(1+h^{-2})\|\nabla\vecb{u}\|_{\vecb{L}^\infty}\|\vecb{\eta}^\mathrm{s}\|^2\Big]\bigg\}\,dt.
\end{aligned}
\end{multline*}
Then \eqref{ieq:velocityestimate22} follows. This completes the proof.
\end{proof}

\section{The reduced scheme}\label{sec:reduced_scheme}
This section discusses a possible condensation of all enrichment
and all higher order pressure degrees of freedom. Here we proceed
similarly to the Stokes case explained in \cite{JLMR2022}.

In order to remove all higher order pressure degrees of freedom
a larger enrichment space $\vecb{V}_h^\mathrm{R}$ has to be used, i.e.,
\begin{align*}
  \vecb{V}_h^\mathrm{R}:=\begin{cases}
  \widetilde{\vecb{RT}}_{k-1}^{\mathrm{int}}(\mathcal{T}) \quad & k\geq d,\\
  (\vecb{RT}_{0}(\mathcal{T})\cap \vecb{H}_0(\mathrm{div},\Omega))\oplus\widetilde{\vecb{RT}}_{k-1}^{\mathrm{int}}(\mathcal{T})\quad & k<d.
  \end{cases}
\end{align*}

The condensation of $\vecb{RT}_0$ DoFs and higher order Raviart--Thomas
bubbles are based on different principles. For the $\vecb{RT}_0$ DoFs,
from \cite[Lemma 3.2]{li2021low} it is not hard to see that the mass
lumping for the $\vecb{RT}_0-\vecb{RT}_0$ block from the time and
diffusion discretizations does not affect the accuracy. Then the
static condensation of $\vecb{RT}_0$ unknowns is possible in case
that the discretization of the nonlinear term does not contribute
to the $\vecb{RT}_0-\vecb{RT}_0$ block, e.g., $c_h(\cdot,\cdot,\cdot)$
is linearized
a with a Picard iteration. For higher order Raviart--Thomas bubbles, based on the
fact that the divergence operator on the chosen higher order enrichment
space is injective, this part of the velocity solution is indeed uniquely
determined by the divergence of the $\vecb{H}^1$ part due to the
divergence constraint. Therefore for each $\vecb{H}^1$ shape function
there is a unique attached Raviart--Thomas function to eliminate the
higher order part of its divergence. Seeking the solution in a subspace
spanned by these Raviart--Thomas attached $\vecb{H}^1$ basis functions
leads to a reduced $\vecb{P}_k-P_0$ scheme, where the higher order
pressure DoFs are no longer needed as Lagrange multipliers. In conclusion,
the condensation of $\vecb{RT}_0$ unknowns is due to the diagonal matrix
structure of the $\vecb{RT}_0-\vecb{RT}_0$ block, while the condensation
of the higher order Raviart--Thomas bubbles, as well as higher order
pressure DoFs, is due to the incompressibility constraint, which is
always available as long as the model is incompressible.
Also note, that the sparsity pattern of the resulting system is the same
as that of a classical $\vecb{P}_k \times P_0$ method for $k \geq d$,
because the interior $\vecb{RT}$ bubbles on one cell do not couple with
the all degrees of freedom of any other cell in the original method.

In what follows the implementation of the reduced scheme on the
algebraic level is shortly discussed.
For simplicity we use the case $k \geq d$ and the case where
the nonlinear term is treated explicitly as an example.

To handle the nonlinear term, one possibility is to put it
explicitly on the right-hand side, such that the matrix stays
constant throughout the whole simulation (if the time step is constant).
Then, algebraically the full linear system in each nonlinear iteration
within a time step has the form
\begin{align*}
  \begin{pmatrix}
    D_{\mathrm{c}\mathrm{c}} + A_{\mathrm{c}\mathrm{c}} & (D_{\mathrm{R}\mathrm{c}} + A_{\mathrm{R}\mathrm{c}})^{\top} & B_\mathrm{c}^{\top} \\
    D_{\mathrm{R}\mathrm{c}} - A_{\mathrm{R}\mathrm{c}} & D_{\mathrm{R}\mathrm{R}} + A_{\mathrm{R}\mathrm{R}} + S_{\mathrm{R}\mathrm{R}}& B_\mathrm{R}^{\top} \\
    B_\mathrm{c} & B_\mathrm{R} & 0 \\
  \end{pmatrix}
  \begin{pmatrix}
    U_\mathrm{c}\\
    U_\mathrm{R}\\
    P
  \end{pmatrix}
  =
  \begin{pmatrix}
    F_\mathrm{c}\\
    F_\mathrm{R}\\
    0
  \end{pmatrix}.
  \end{align*}
  Here, the $A$-blocks refer to the linear Stokes operators, the
  $D$-blocks refer to the mass matrix and $F$ contains all
  right-hand side terms from $f$, the nonlinear term and the
  previous time step. The matrix $S_{RR}$ refers to the used
  stabilization $S_1$ or $S_2$ for the enrichment part.

  Following \cite{JLMR2022} we employ a reconstruction operator
  $\mathcal{R} : \vecb{V}_h^\mathrm{ct} \rightarrow \vecb{V}_h^\mathrm{R} $
  and its representation matrix $R$. Recall,
  that the enrichment part is uniquely determined by
  $U_\mathrm{R} = - R U_\mathrm{c}$.
  With that, the problem can be reduced to solving
  \begin{multline*}
    \begin{pmatrix}
      D_{\mathrm{c}\mathrm{c}} + A_{\mathrm{c}\mathrm{c}} + (D_{\mathrm{R}\mathrm{c}} - A_{\mathrm{R}\mathrm{c}})^{\top}R + R^{\top}(D_{\mathrm{R}\mathrm{c}} + A_{\mathrm{R}\mathrm{c}}) & B_\mathrm{0,c}^{\top} \\
      B_\mathrm{0,c} & 0
    \end{pmatrix}
    \begin{pmatrix}
      U_\mathrm{c}\\
      P_0
    \end{pmatrix}
    =
    \begin{pmatrix}
      F_\mathrm{c} - R^{\top} F_\mathrm{R}\\
      0
    \end{pmatrix}.
    \end{multline*}
    Here, $P_0$ refers to the piecewise constant pressure
    and $B_{0,\mathrm{ct}}$ to the div-pressure matrix
    between velocities of $\vecb{V}_h^\mathrm{c}$ and $P_0$ pressures.
    Also the full pressure can be recovered by solving small local problems.
    All details can be found in \cite{JLMR2022} as there are no additional
    difficulties if the nonlinear term is handled explicitly in the
    right-hand side.

\section{Numerical Examples}

This section demonstrates the performance of the proposed schemes
in two benchmark examples. Concerning the nonlinear solver
a Picard iteration scheme was employed and terminated
when the difference to the last iteration
in the $H^1$-norm is below the tolerance $10^{-8}$.
The enrichment spaces $\vecb{V}_h^\mathrm{R}$ are built from the
basis functions constructed in \cite{JLMR2022}.

\label{sec:num}
\subsection{Example 1}
The first example from \cite[Example 4.2]{GJN21} takes the inhomogeneous data such that it matches the flow
\begin{align*}
  \vecb{u}(x,y,t) & = 2 \pi \sin(\pi t)
  \begin{pmatrix}
    \sin(\pi x)^2 \sin(\pi y) \cos(\pi y) \\
    -\sin(\pi y)^2 \sin(\pi x) \cos(\pi x)
  \end{pmatrix},\\
    p(x,y,t) & = 20 \sin(\pi t)(x^2 y - 1/6)
\end{align*}
for $\nu = 10^{-6}$ and computes the discrete $\vecb{u}_h$ and $p_h$ with the Crank--Nicolson scheme with time step $\tau = 10^{-4}$ until $T = 3$
unstructured meshes, the initial mesh is depicted in Figure~\ref{fig:ex1_meshlevel0}. The parameter of $\mathcal{S}$, $\gamma$, is taken as 1 always.
Figures~\ref{fig:ex1_convratep2}-\ref{fig:ex1_convratep4} display results for orders $k \in \lbrace 2,3,4 \rbrace$
for the maximal $L^2$ error in time and the full estimated norm in Theorems~\ref{thm:velocityestimate} and \ref{thm:velocityestimate2}, i.e,
\begin{align*}
  ||| \vecb{u} |||^2_\star
  := \| \vecb{u}^\mathrm{s} \|^2 + \nu \int_0^T \| \nabla \vecb{u}^\mathrm{ct} \|^2 dt
  + \begin{cases}
     \| \vecb{u}^\mathrm{R} \|^2 & \text{ for } \eqref{eq:semischeme2} \text{ with }\mathcal{S}_1,\\
     \int_0^T \| h_\mathcal{T}^\frac{1}{2} \nabla \cdot \vecb{u}^\mathrm{R} \|^2 dt & \text{ for } \eqref{eq:semischeme2} \text{ with } \mathcal{S}_2,\\
     \int_0^T \lvert \vecb{u} \rvert^2_{\textbf{u}_h, \mathrm{uw}} dt & \text{ for upwind DG.}
  \end{cases}
\end{align*}
It can be seen that all stabilizations attain at least their estimated
convergence orders and that the upwind stabilization yields the best results
among the three tested stabilizations. The upwind stabilization even shows
a better convergence rate pre-asymptotically.
The easier to implement stabilization $\mathcal{S}_2$ yields the
second-best results. Surprisingly, the asymptotic convergence rates
for $k=3$ are much better than expected from the theory.

\begin{figure}
  \centering
  \includegraphics[width=0.49\textwidth]{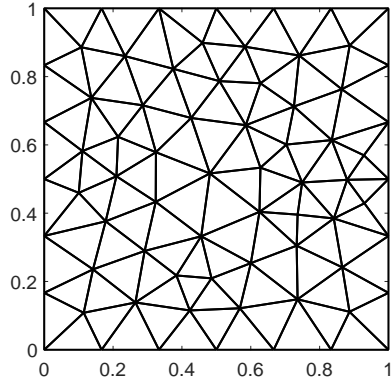}
  \caption{\label{fig:ex1_meshlevel0}Initial grid for Example 1.}
\end{figure}

\begin{figure}
  \includegraphics[width=0.49\textwidth]{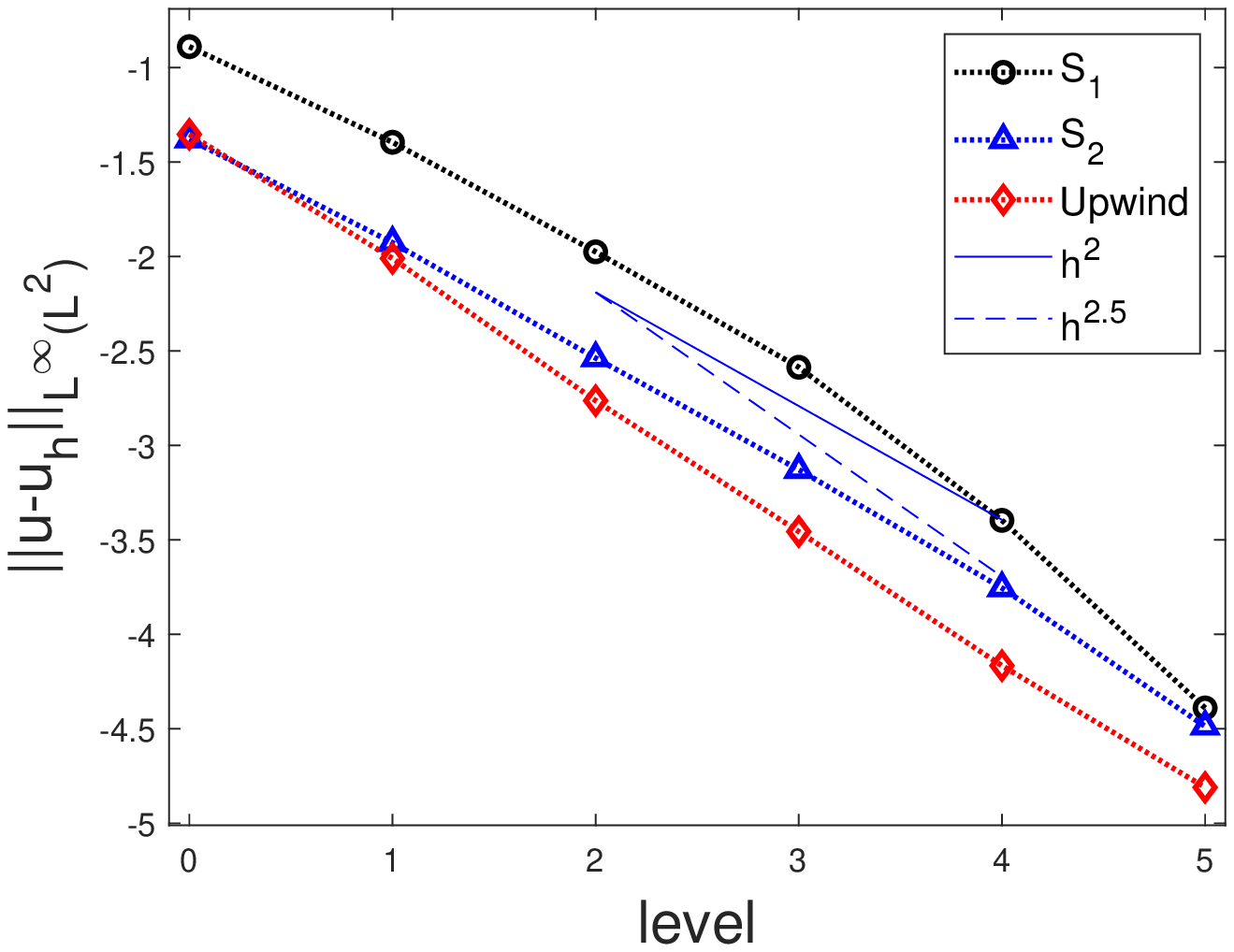}
  \includegraphics[width=0.49\textwidth]{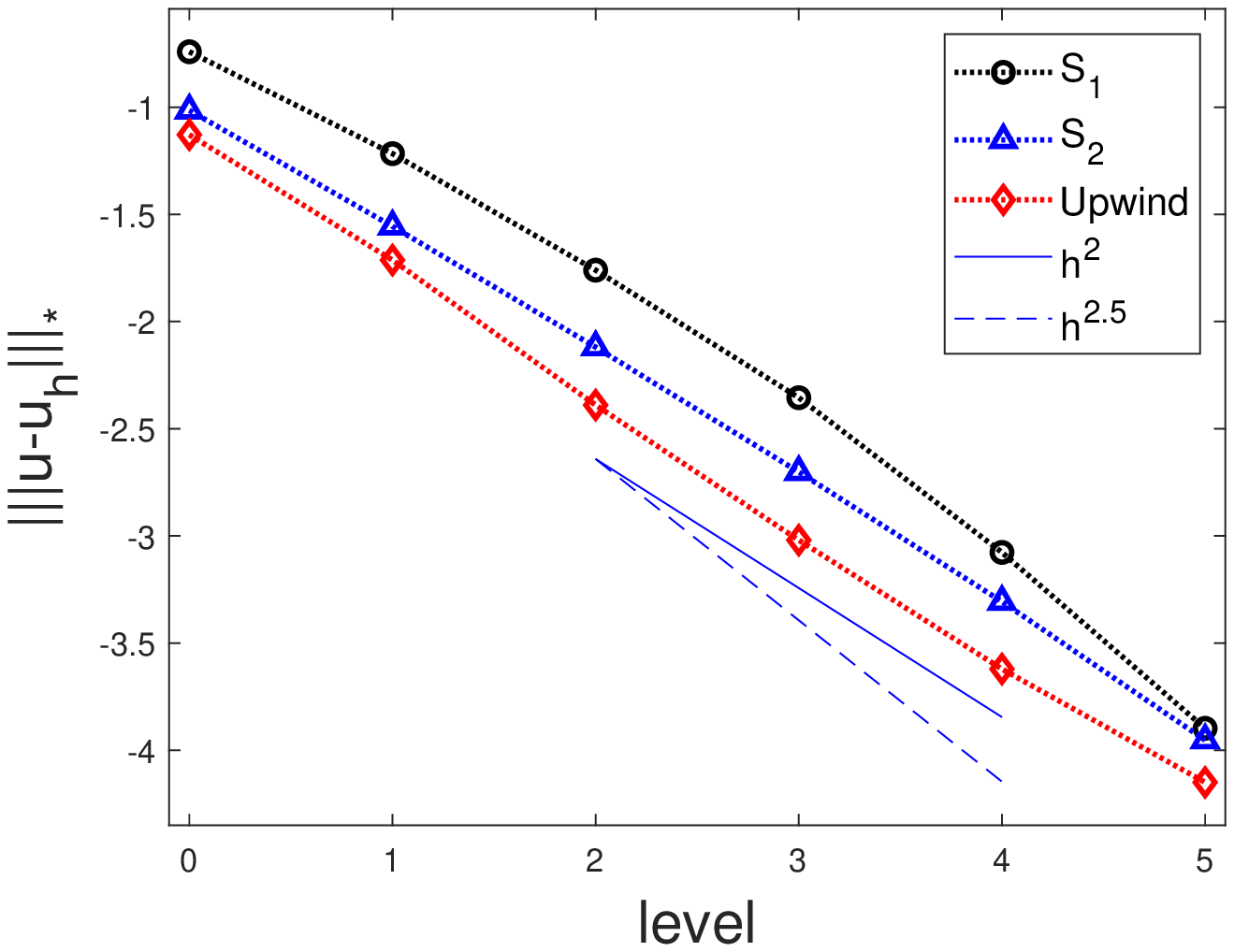}
  \caption{\label{fig:ex1_convratep2}Example 1: Convergence rates of $\|\vecb{u}-\vecb{u}_h^\mathrm{s}\|_{L^\infty(L^2)}$ (energy error, left) and $|||(\vecb{u},0)-\vecb{u}_h|||_\star$ (right) from the estimate with $\nu=10^{-6}$, $\Delta t = 10^{-4}$, and $T=3$ for the reduced $P_2-P_0$ scheme.}
\end{figure}
\begin{figure}
  \includegraphics[width=0.49\textwidth]{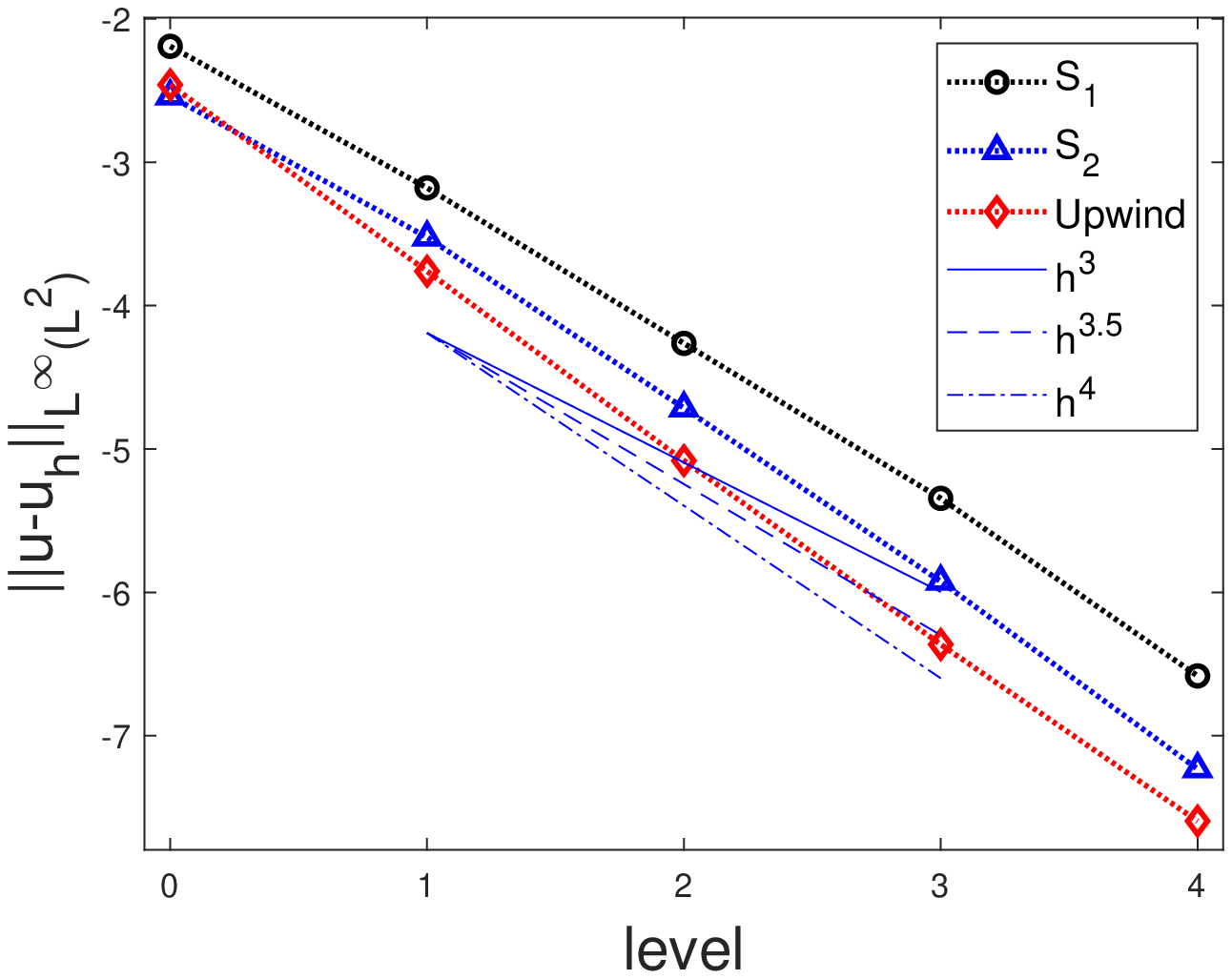}
  \includegraphics[width=0.49\textwidth]{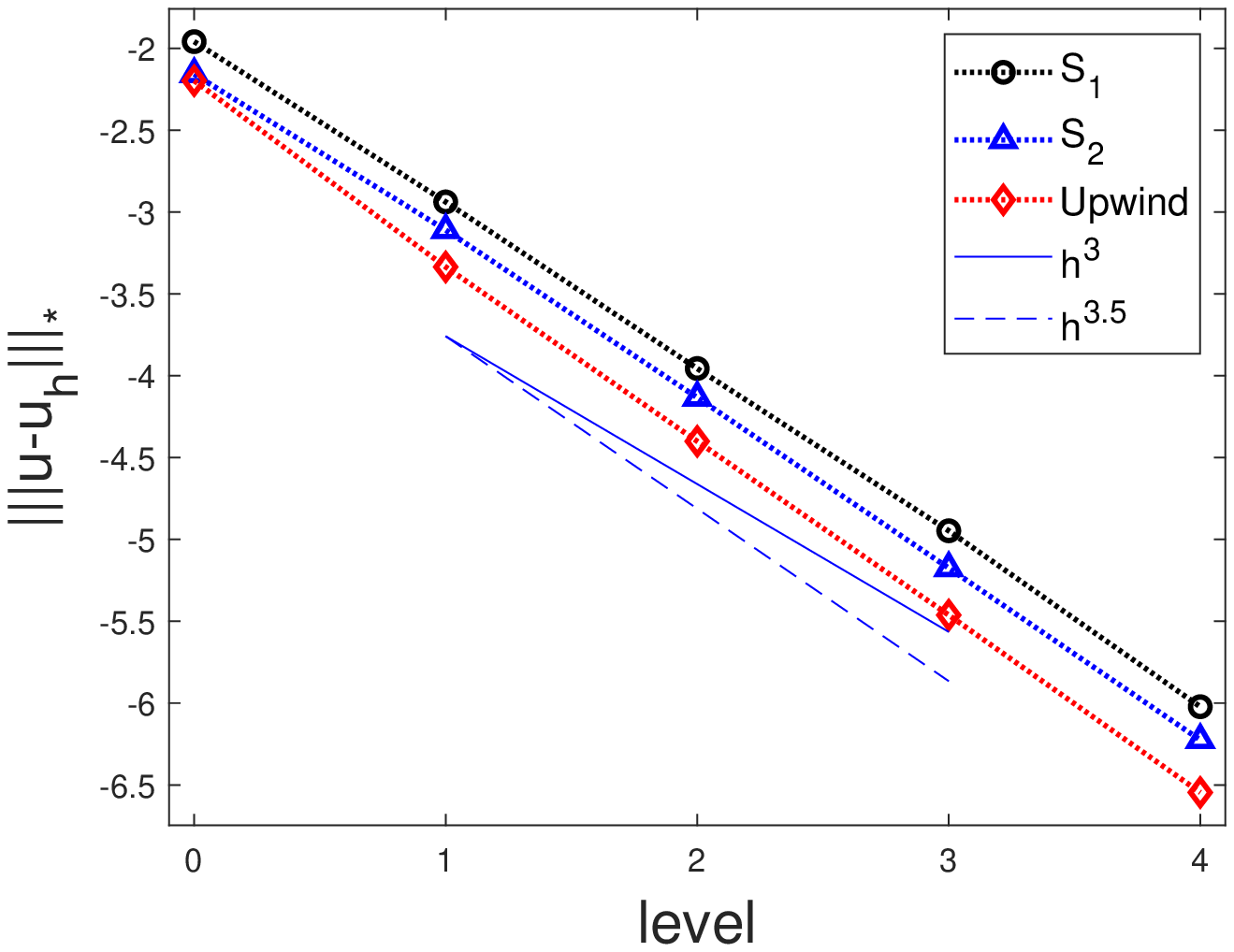}
  \caption{\label{fig:ex1_convratep3}Example 1: Convergence rates of $\|\vecb{u}-\vecb{u}_h^\mathrm{s}\|_{L^\infty(L^2)}$ (energy error, left) and $|||(\vecb{u},0)-\vecb{u}_h|||_\star$ (right) from the estimate with $\nu=10^{-6}$, $\Delta t = 10^{-4}$, and $T=3$ for the reduced $P_3-P_0$ scheme.}
\end{figure}
\begin{figure}
  \includegraphics[width=0.49\textwidth]{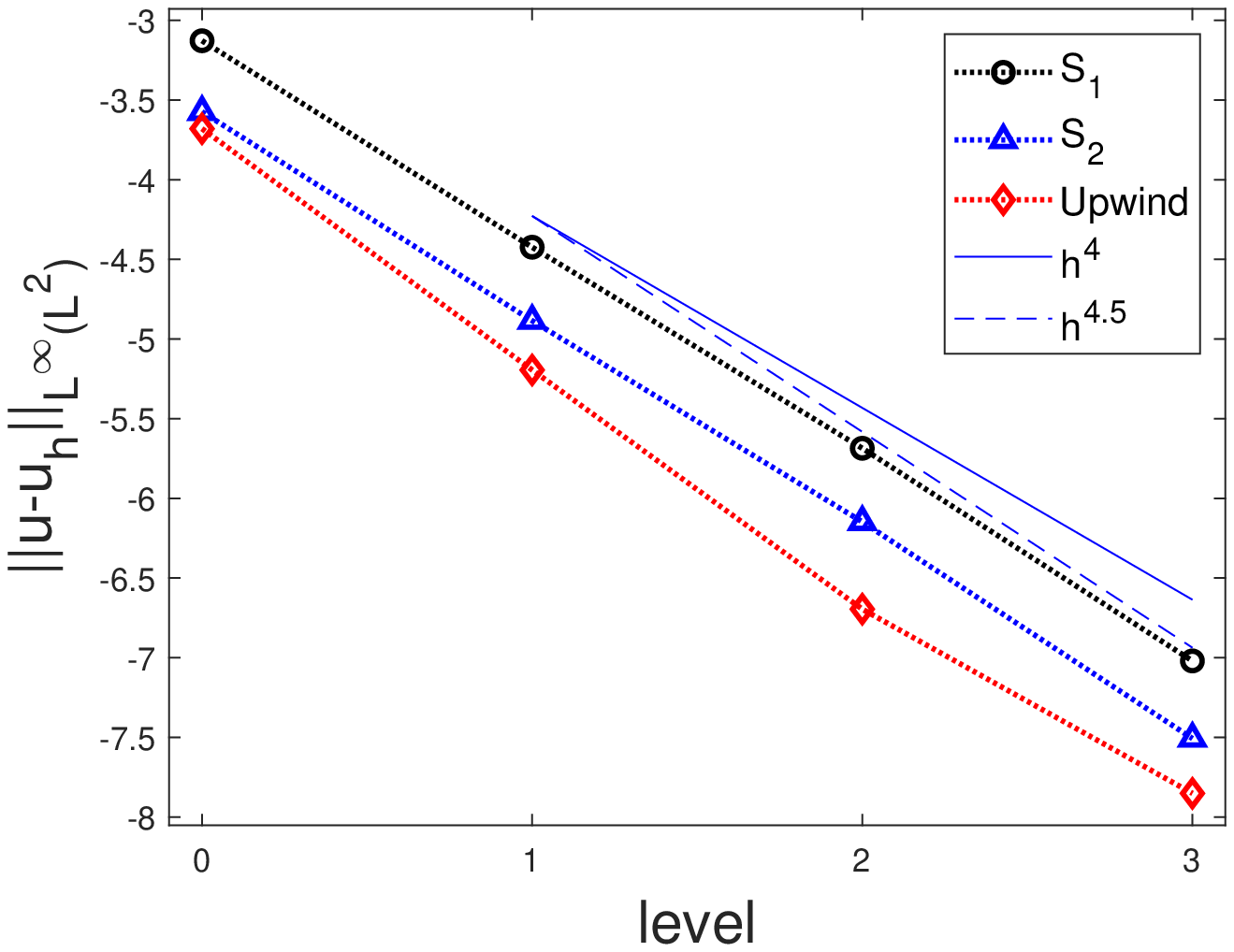}
  \includegraphics[width=0.49\textwidth]{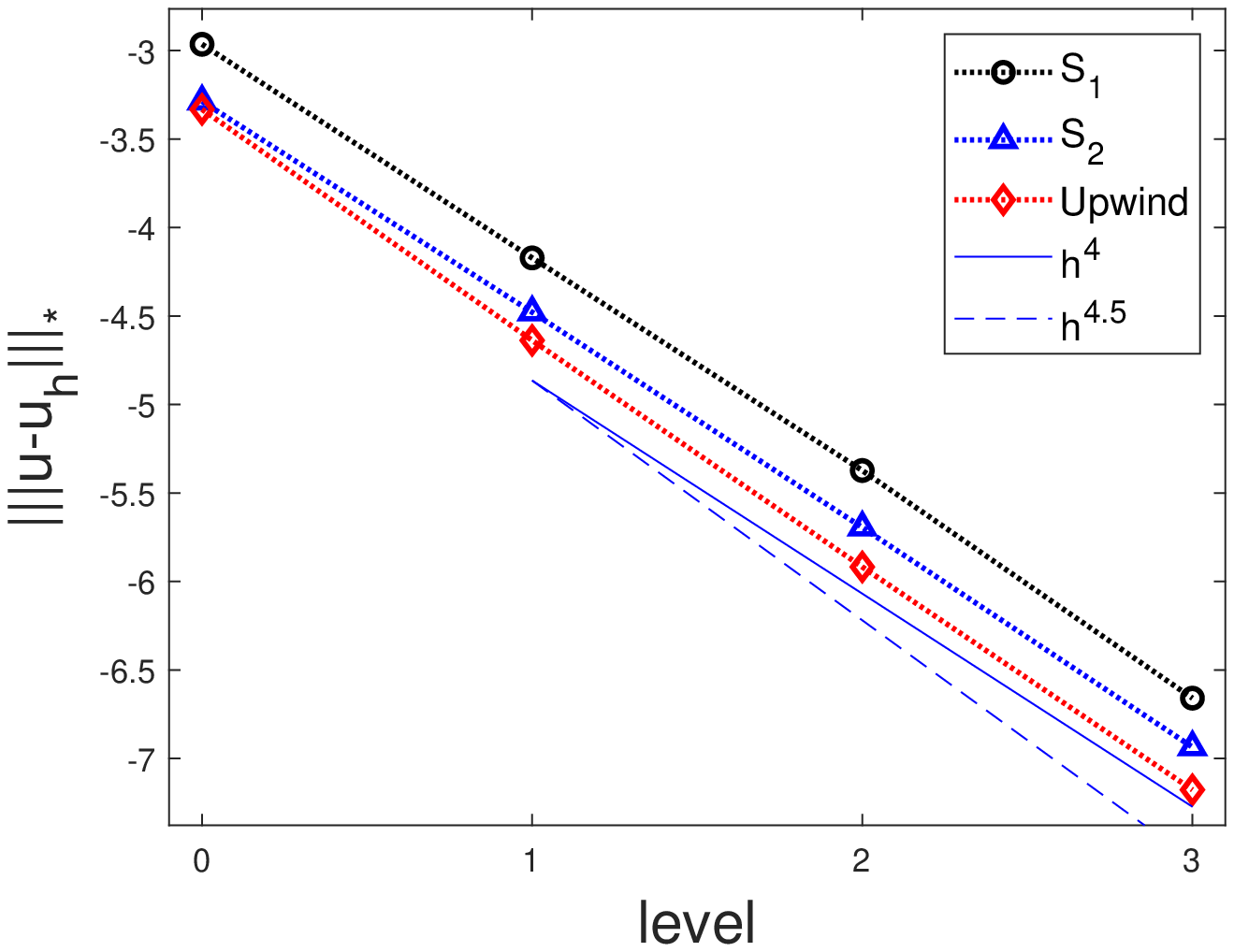}
  \caption{\label{fig:ex1_convratep4}Example 1: Convergence rates of $\|\vecb{u}-\vecb{u}_h^\mathrm{s}\|_{L^\infty(L^2)}$ (energy error, left) and $|||(\vecb{u},0)-\vecb{u}_h|||_\star$ (right) from the estimate with $\nu=10^{-6}$, $\Delta t = 10^{-4}$, and $T=3$ for the reduced $P_4-P_0$ scheme.}
\end{figure}
\begin{figure}
  \includegraphics[width=0.49\textwidth]{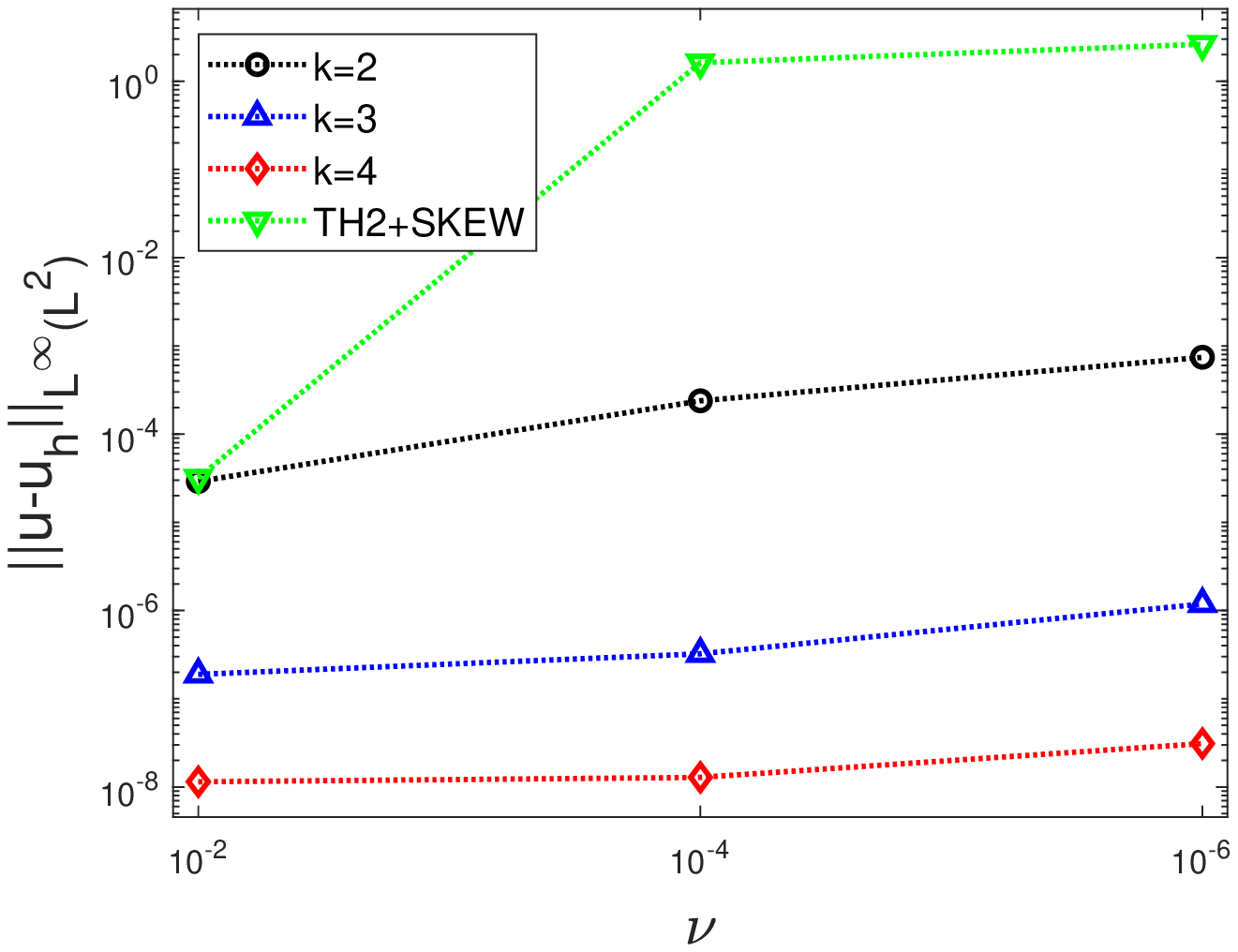}
  \includegraphics[width=0.49\textwidth]{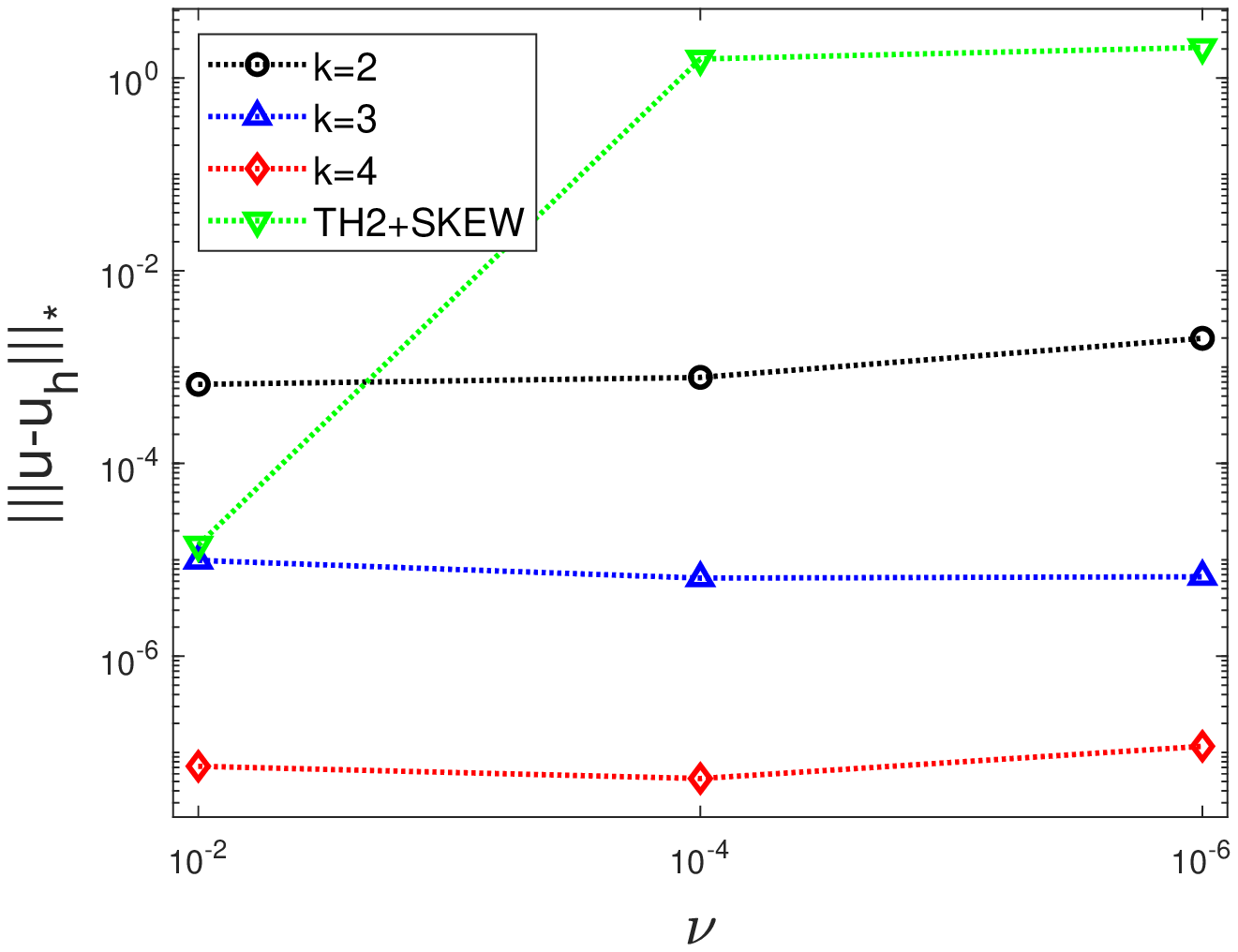}
  \caption{\label{fig:ex1_convrobust}Example 1: Plots of $\|\vecb{u}-\vecb{u}_h^\mathrm{s}\|_{L^\infty(L^2)}$ (energy error, left) and $|||(\vecb{u},0)-\vecb{u}_h|||_\star$ (right) from the estimate over $\nu\in\{10^{-2},10^{-4},10^{-6}\}$ with $\Delta t = 10^{-4}$ and $T=3$ for the reduced and $\mathcal{S}_2$ stabilized $P_k-P_0$ schemes ($k\in\{2,3,4\}$) on mesh level 3.}
\end{figure}

Figure~\ref{fig:ex1_convrobust} provides some study on convection-robustness
and depicts the maximal $L^2$ error for the reduced schemes of order
$k \in \lbrace 2,3,4 \rbrace$ with stabilization
$\mathcal{S}_2$ and a second-order classical Taylor--Hood
method with skew-symmetric discretization of the nonlinear convection term
for different choices of $\nu \in \lbrace 10^{-2}, 10^{-4},10^{-6} \rbrace$.
Here, two observations are in order. First, the Raviart--Thomas enriched methods
are convection-robust for all three stabilizations in the sense that the errors are
relatively insensitive to $\nu$. Second, the Taylor--Hood method
shows a large increase in the error when going from $\nu=10^{-2}$
to $\nu = 10^{-4}$ which indicates that this method is not
convection-robust.
Note however, that the Taylor--Hood method can be improved by adding
grad-div stabilization to have similar results and was just added in this form
to show the behavior of a non-convection-robust scheme.

Finally, Table~\ref{tab:costcomparison} compares the numerical costs for
the full and reduced scheme with $\mathcal{S}_2$ stabilization
and a Taylor--Hood scheme
with skew-symmetric convection term (SKEW) and grad-div stabilization
of the same order for $k \in \lbrace 2,3,4 \rbrace$. This time, the grad-div stabilization was
added such that all methods do exactly two implicit Picard iterations in each time
step (so altogether $10.000$ Picard iterations for $5.000$ time steps).
The direct
sparse matrix solver Pardiso \cite{SGFSpardiso} is used. For the first
iteration in each of the first 7 time steps, a complete solver (symbolic
factorization + numerical factorization + forward and backward substitutions) is employed. The
simulation was done on a laptop with an 11th Gen Intel(R) Core(TM)
i5-11400H CPU with 6 cores. For all other iterations we keep the
symbolic factorization of the sparse matrix.
The table reports the number of
degrees of freedom, nonzeros in the system matrix and solver
times in each case. It can be seen, that the reduction procedure
significantly reduces the numerical costs and that the full scheme,
despite the larger number of degrees of freedom,
has comparable numerical costs and solver times to a Taylor--Hood scheme.
\begin{table}
\caption{Example 1: Number of DoFs, number of nonzero entries of
the coefficient matrix $A$, and cost of the proposed method
(with $\mathcal{S}_2$ stabilization) and Taylor--Hood element method
(with SKEW and grad-div stabilization scaled by $h$) on mesh level 3
(7424 elements), with $\Delta t= 10^{-4}$ and $T=0.5$
(which amounts to totally 5000 time steps and 10000 Picard
iterations in every method).
$\text{Cost}_1$: total computing time in seconds.
$\text{Cost}_2$: total computing time in seconds on Pardiso.
$\text{Cost}_3$: mean computing time in seconds per iteration where a
complete solver is used.}
\footnotesize
\begin{tabular}{llccccccc}
\hline Order & Method & $\#\{\vecb{u}$ DoFs$\}$ & $\#\{p$ DoFs$\}$ & $\#\left\{\mathrm{nz}\left(A\right)\right\}$ & $\text{Cost}_1$ & $\text{Cost}_2$ & $\text{Cost}_3$ \\
\hline $k=2$ & Full & $44930$ & $22272$ & $1718 \mathrm{~K}$ & $2.31\text{e}3$ & $1.23\text{e}3$ & 0.5804 \\
             & Reduced & $30082$ & $7424$ & $872 \mathrm{~K}$ & $1.78\text{e}3$ & $8.44\text{e}2$ & $0.3777$ \\
             & TH & $30082$ & $3809$ & $974 \mathrm{~K}$ & $1.85\text{e}3$ & $1.07\text{e}3$ & $0.4213$ \\
       $k=3$ & Full & $89666$ & $44544$ & $5329 \mathrm{~K}$ & $5.50\text{e}3$ & $2.81\text{e}3$ & $1.4144$ \\
             & Reduced & $67394$ & $7424$& $2582 \mathrm{~K}$ & $4.22\text{e}3$ & $2.06\text{e}3$ & $0.7799$ \\
             & TH & $67394$ & $15041$ & $3559 \mathrm{~K}$ & $5.86\text{e}3$ & $3.81\text{e}3$ & $1.2671$ \\
       $k=4$ & Full & $149250$ & $74240$ & $12615 \mathrm{~K}$ & $1.12\text{e}4$ & $5.57\text{e}3$ & $2.8217$ \\
             & Reduced & $119554$ & $7424$& $6045 \mathrm{~K}$ & $8.20\text{e}3$ & $3.80\text{e}3$ & $1.5864$ \\
             & TH & $119554$ & $33697$ & $9211 \mathrm{~K}$ & $1.39\text{e}4$ & $9.17\text{e}3$ & $2.9569$ \\
\hline
\end{tabular}
\label{tab:costcomparison}
\end{table}

\subsection{Example 2}

This example considers the Kelvin--Helmholtz instability benchmark problem for which
reference values are computed and discussed in \cite{Schroeder_2019}.
Here, the case of Reynolds number $Re = 10.000$ and
the initial condition
\begin{align*}
  \vecb{u}^0(x,y) =
  \begin{pmatrix}
    u_\infty \tanh(\frac{2y-1}{\delta_0})\\
    0
  \end{pmatrix}
  + c_n
  \begin{pmatrix}
    -\partial_y \psi(x,y)\\
    \partial_x \psi(x,y)
  \end{pmatrix}
\end{align*}
with the stream function
\begin{align*}
  \psi(x,y) = u_\infty \exp\left(-\frac{(y-0.5)^2}{\delta_0^2}\right) (\cos(8\pi x) + \cos(20\pi x))
\end{align*}
is studied, where the parameters are chosen to be
$\delta_0 = 1/28$, $u_\infty = 1$ and $c_n = 10^{-3}$.
Via the relation $Re = \delta_0 u_\infty / \mu$,  the viscosity
is $\mu = 1 / (Re \delta_0) = 1/280.000 \approx 3.57 \cdot 10^{-6}$.
At the left and right boundary, i.e., for $x = 0$ and $x = 1$, periodic conditions are applied, while
at the top and bottom boundary, i.e., for $y = 0$ and $y = 1$, free-slip conditions are applied.

\begin{figure}\hfill
  \includegraphics[width=0.4\textwidth]{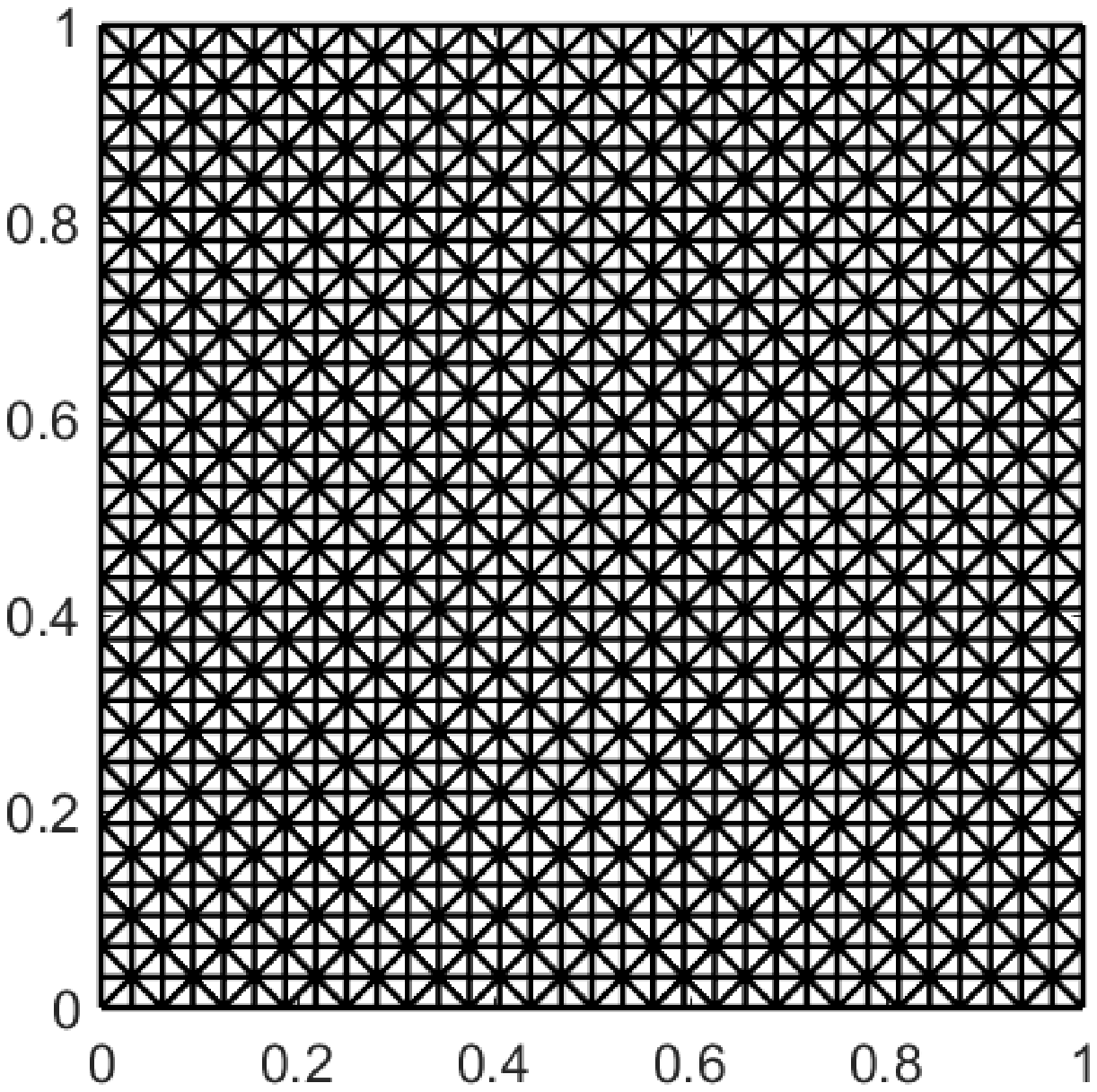}
  \hfill
  \includegraphics[width=0.4\textwidth]{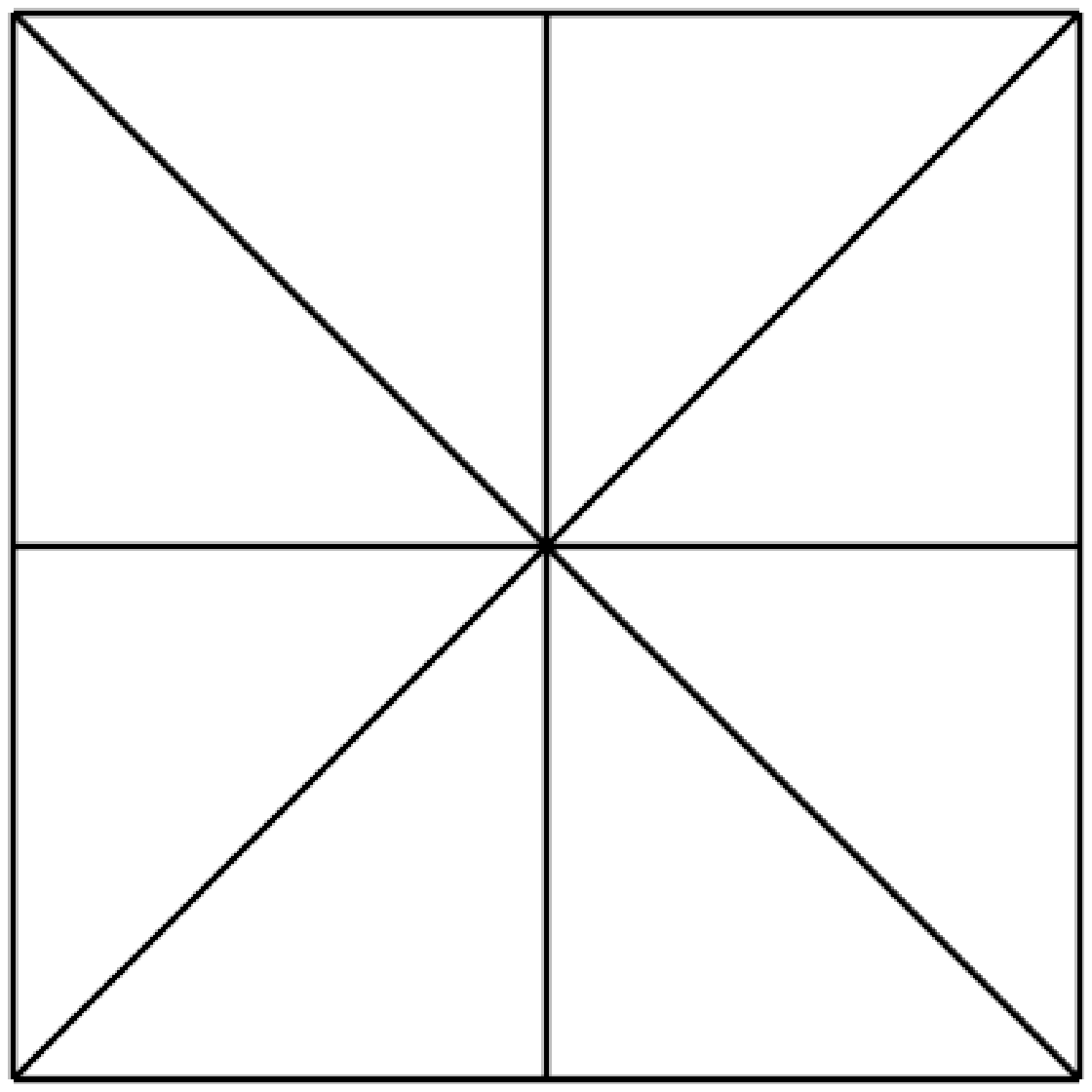}\hfill\hfill
  \caption{\label{fig:KHmeshes}Example 2: A mesh consisting of $32 \times 32$ divided squares (left) built
  by repeating a $2 \times 2$ squared structure (right).}
\end{figure}

This problem describes the successive pairing of vortices until finally only one rotating vortex remains.
It is clearly convection-dominated so that a stabilized method, e.g., a convection-robust method, is
necessary for performing stable numerical simulations. In \cite{Schroeder_2019} simulations were performed
on grids up to $256\times 256$ mesh cells and with a velocity space of polynomial degree $8$. It was found that
predicting a reference solution is quite challenging for the Reynolds number used in our simulations. In fact,
a reference solution could be obtained only up to $t/\delta_0 = 200$, i.e., $t=50/7 \approx 7.14$, which corresponds to the situation
of two rotating vortices. It has been observed in \cite{Schroeder_2019} that the formation time of the
final vortex and  also its position is very sensitive to
the setup and even tiny differences in algorithms (or even compiler options) can lead to a different behavior.

The example is simulated on structured and symmetric meshes like the one depicted in Figure~\ref{fig:KHmeshes}
and for the polynomial order $k = 2$.
For the time discretization, we employ the
Crank--Nicolson scheme with time step $\tau = 10^{-3} \delta_0$.
Quantities of interest, besides the kinetic energy, comprise the enstrophy given by
\[
   \mathcal{E}(t, \vecb{u}) := \frac{1}{2} \| \mathrm{Curl}(\vecb{u}) \|_{L^2(\Omega)}
   \quad \mbox{with} \quad \mathrm{Curl}(\vecb{u}) := \partial_x u_2 - \partial_y u_1,
\]
which should by monotonously decreasing in time, and
the vorticity thickness defined by
\[
    \delta(t, \vecb{u}) := \frac{2 u_\infty}{\sup_{y \in [0,1]} \left| \int_0^1 \mathrm{Curl}(\vecb{u})\ \textit{dx}\right|},
\]
where we approximate the supremum in the denominator by taking the maximum
over all $y \in \lbrace (1+2k)/2048 : k = 1,\ldots, 1024 \rbrace$. In this example we try scheme \eqref{eq:semischeme2} with stabilization $\mathcal{S}_2$ only and the parameter $\gamma$ is chosen as 1 again.

The merging process of the vortices obtained from our method is shown in Figure~\ref{fig:KH_vorticity}.
Figure~\ref{fig:KH_QOI} shows the evolution of these quantities in the performed
simulation.
Although our simulations were performed on relatively coarse meshes and
low polynomial order of the velocity, we observe that the general
qualitative behavior of the vorticity thickness and the enstrophy is
inline with the expectations. The vorticity thickness shows the right kinks when
the vortices are forming. Even for $t/\delta_0 > 200$, where no reference
solution is available, the qualitative behavior of the results of
the proposed scheme corresponds with those from \cite{Schroeder_2019}.

\begin{figure}
  \centering

  \includegraphics[width=0.30\textwidth]{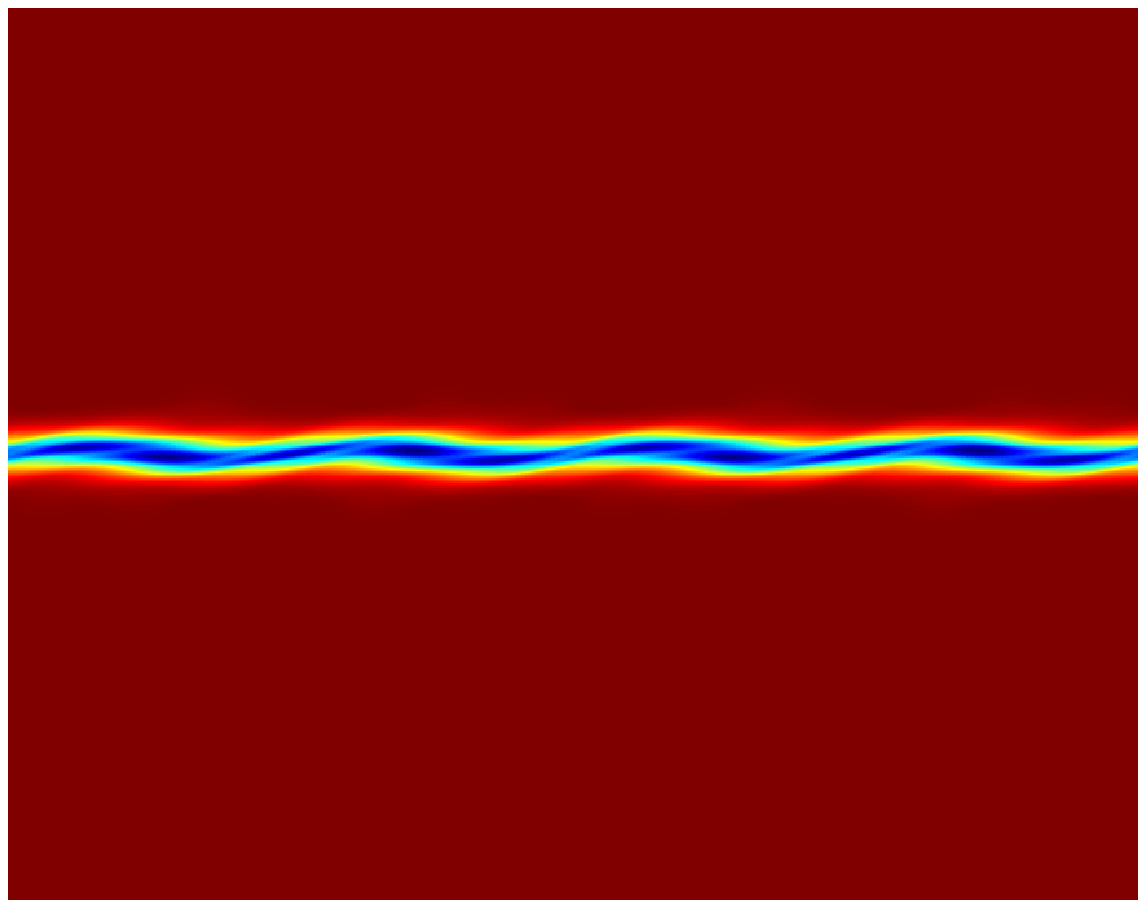}
  \includegraphics[width=0.30\textwidth]{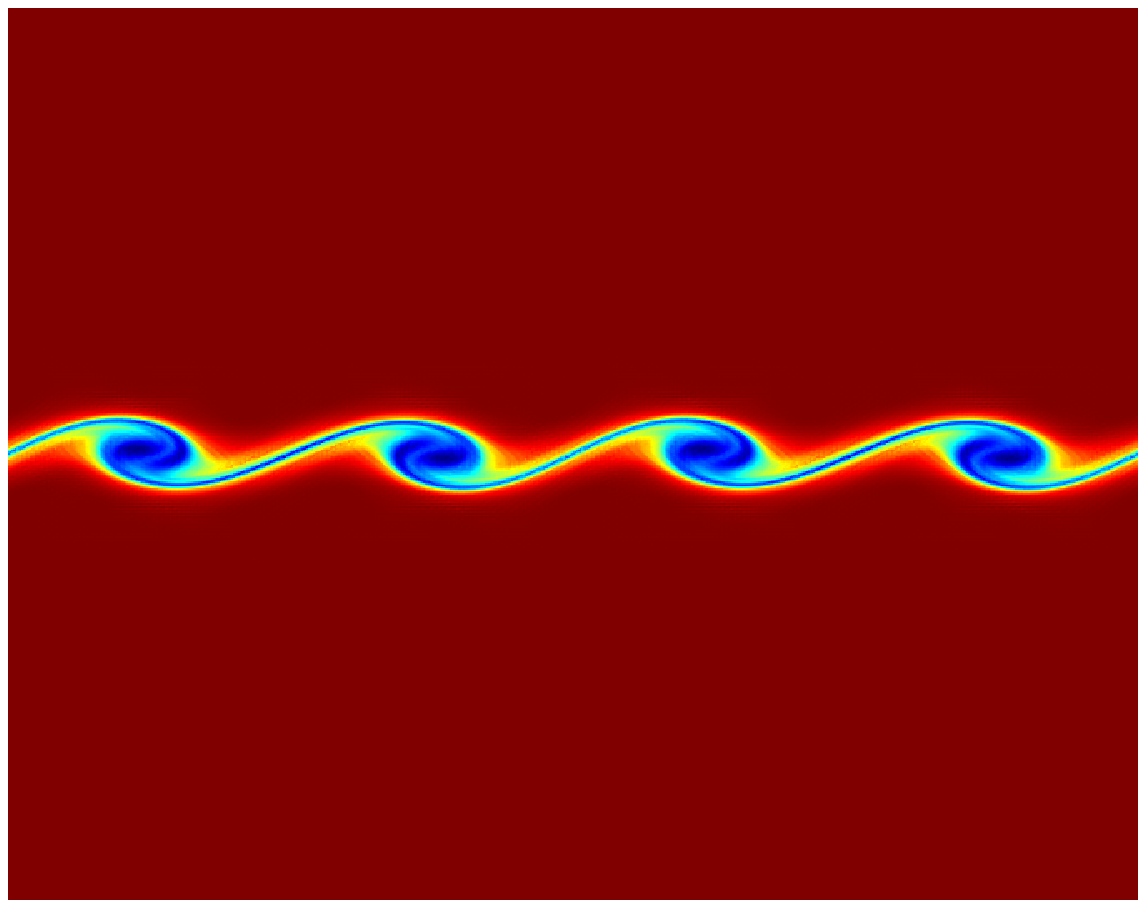}
  \includegraphics[width=0.30\textwidth]{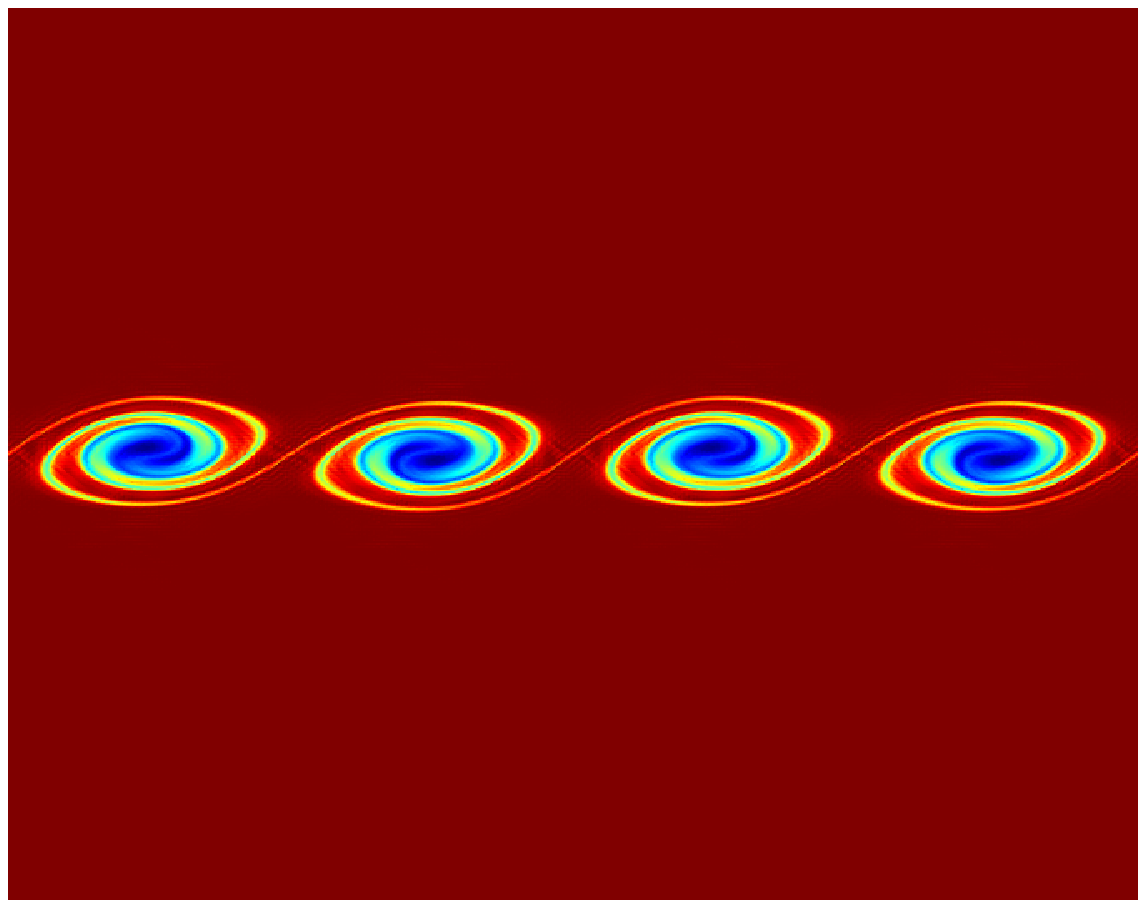}

  \includegraphics[width=0.30\textwidth]{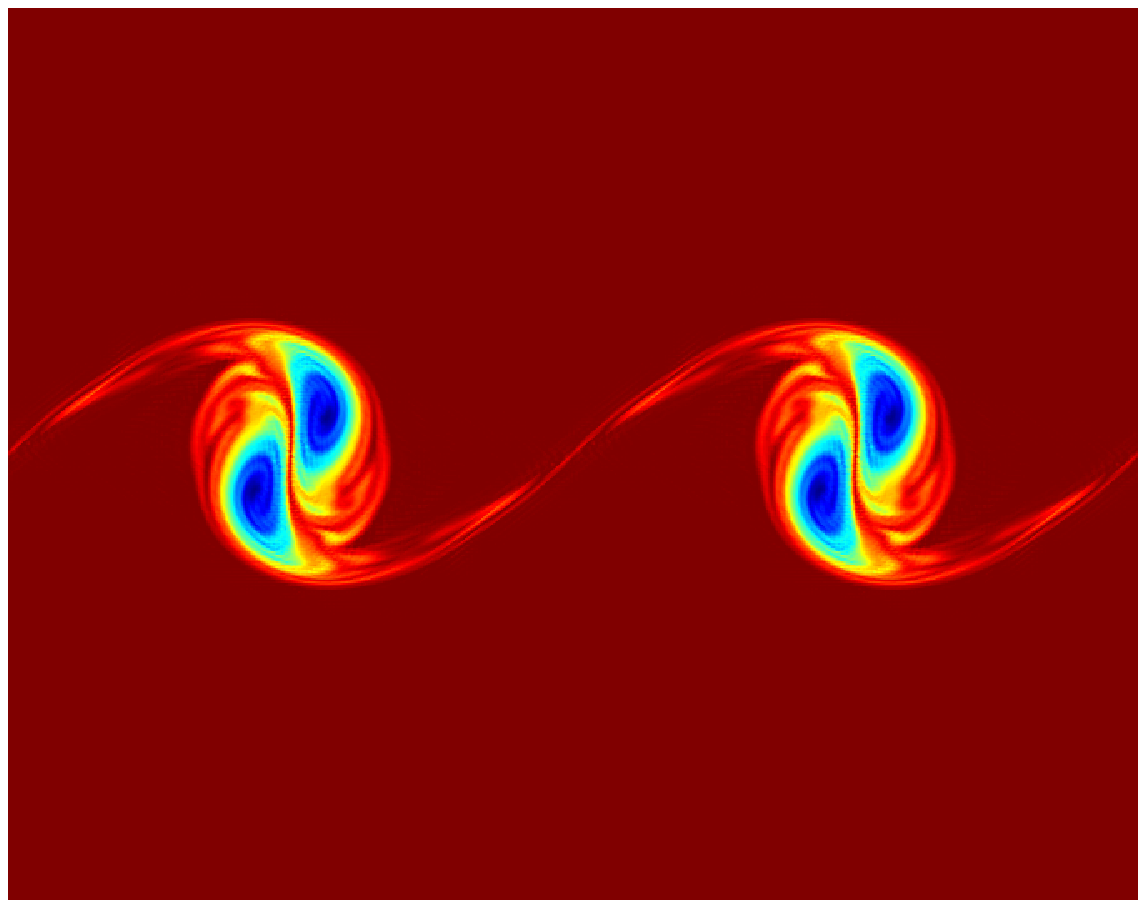}
  \includegraphics[width=0.30\textwidth]{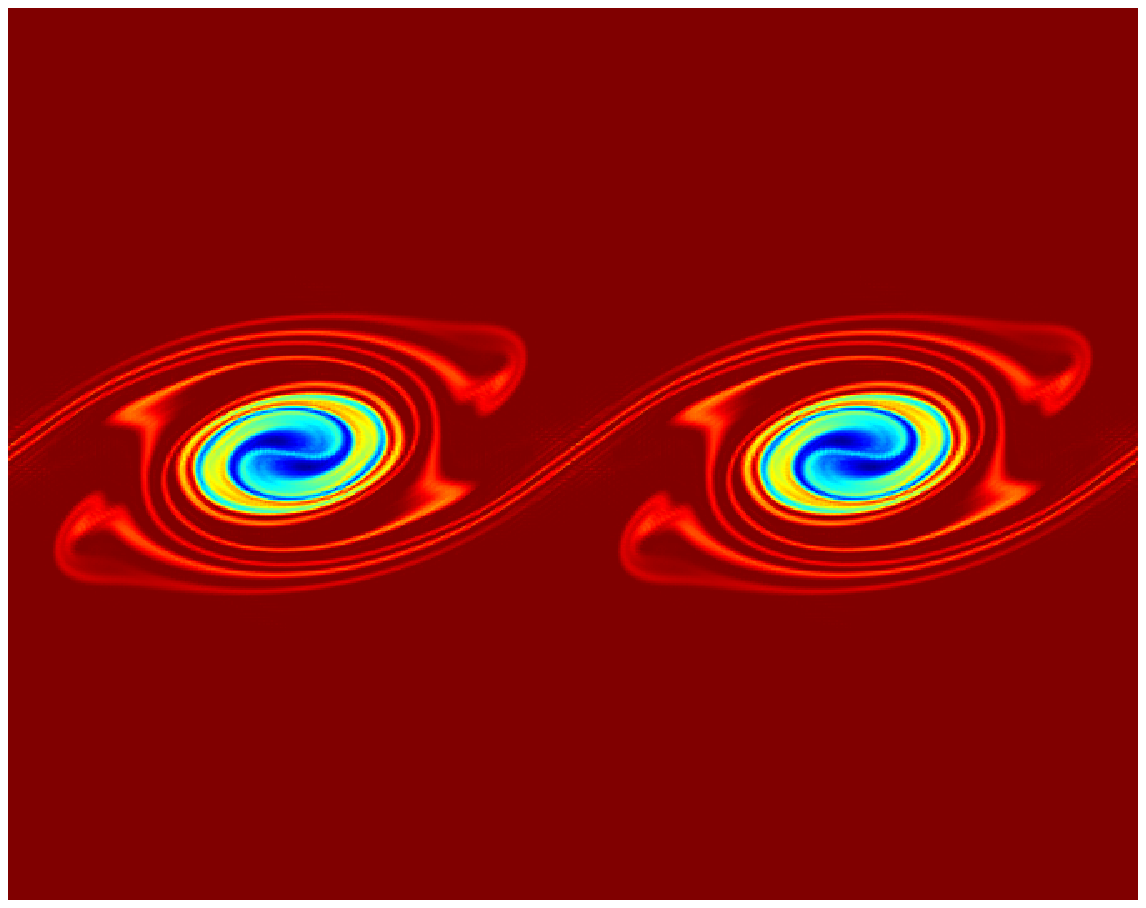}
  \includegraphics[width=0.30\textwidth]{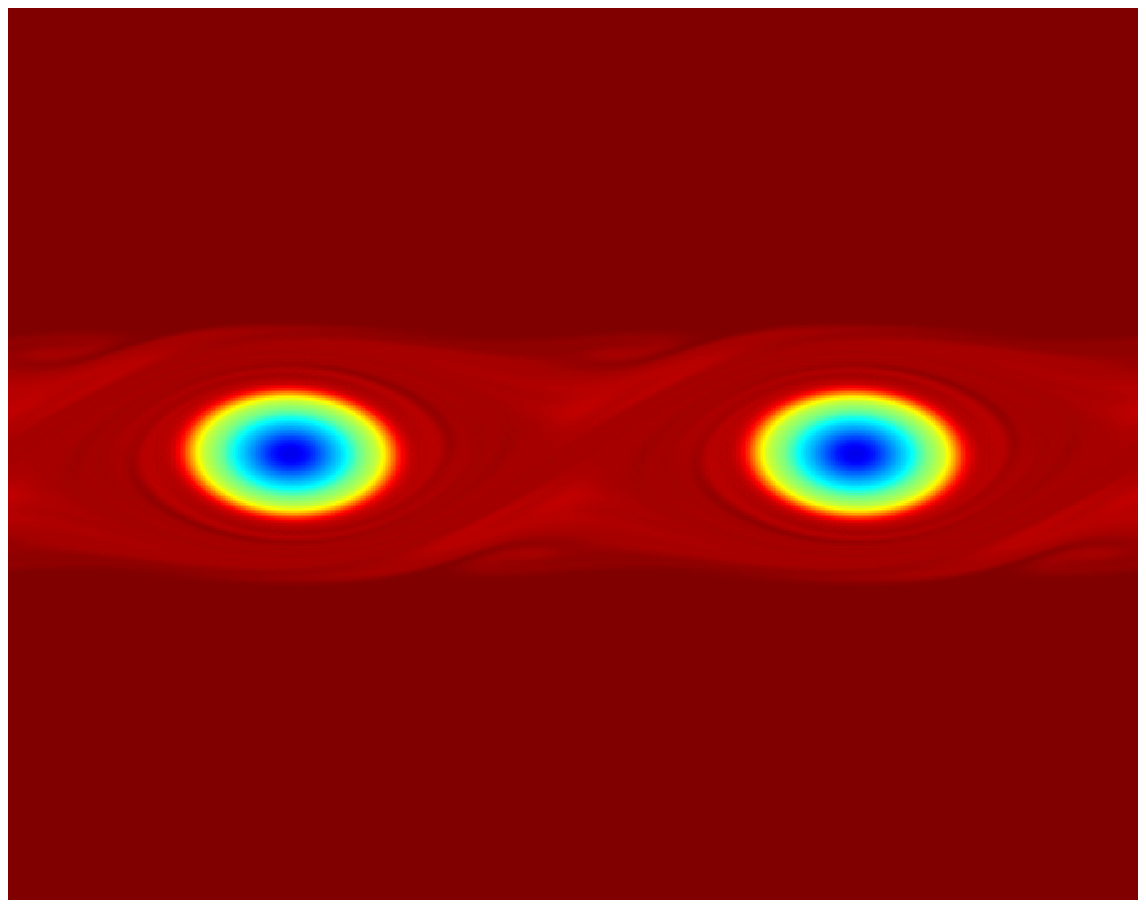}

  \includegraphics[width=0.30\textwidth]{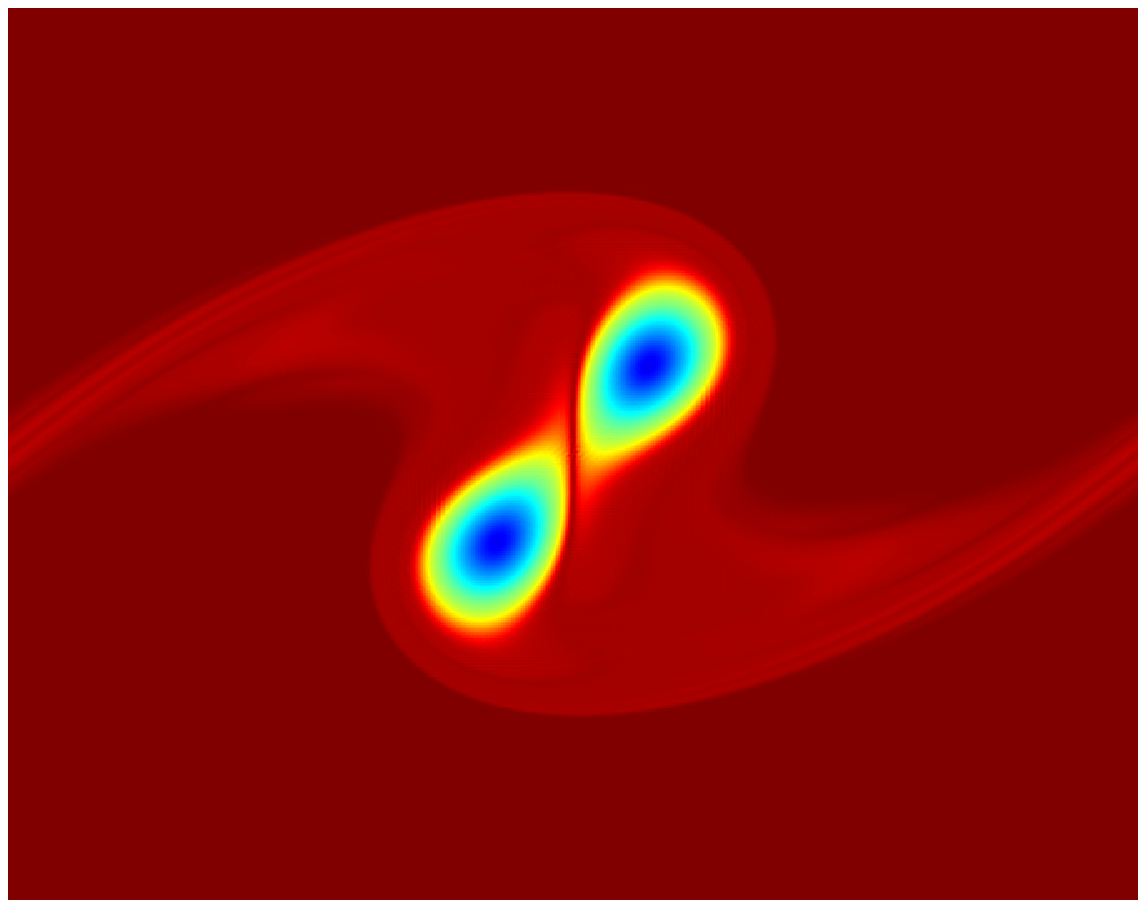}
  \includegraphics[width=0.30\textwidth]{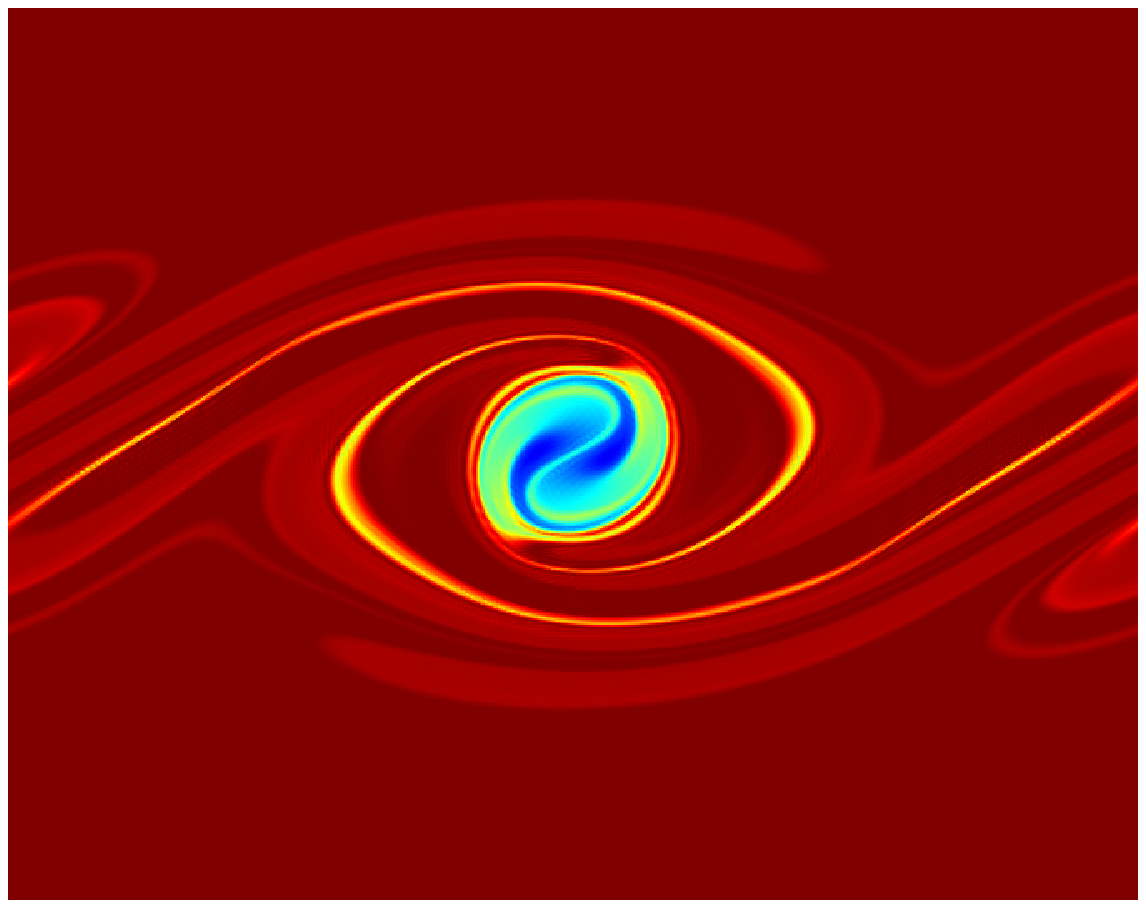}
  \includegraphics[width=0.30\textwidth]{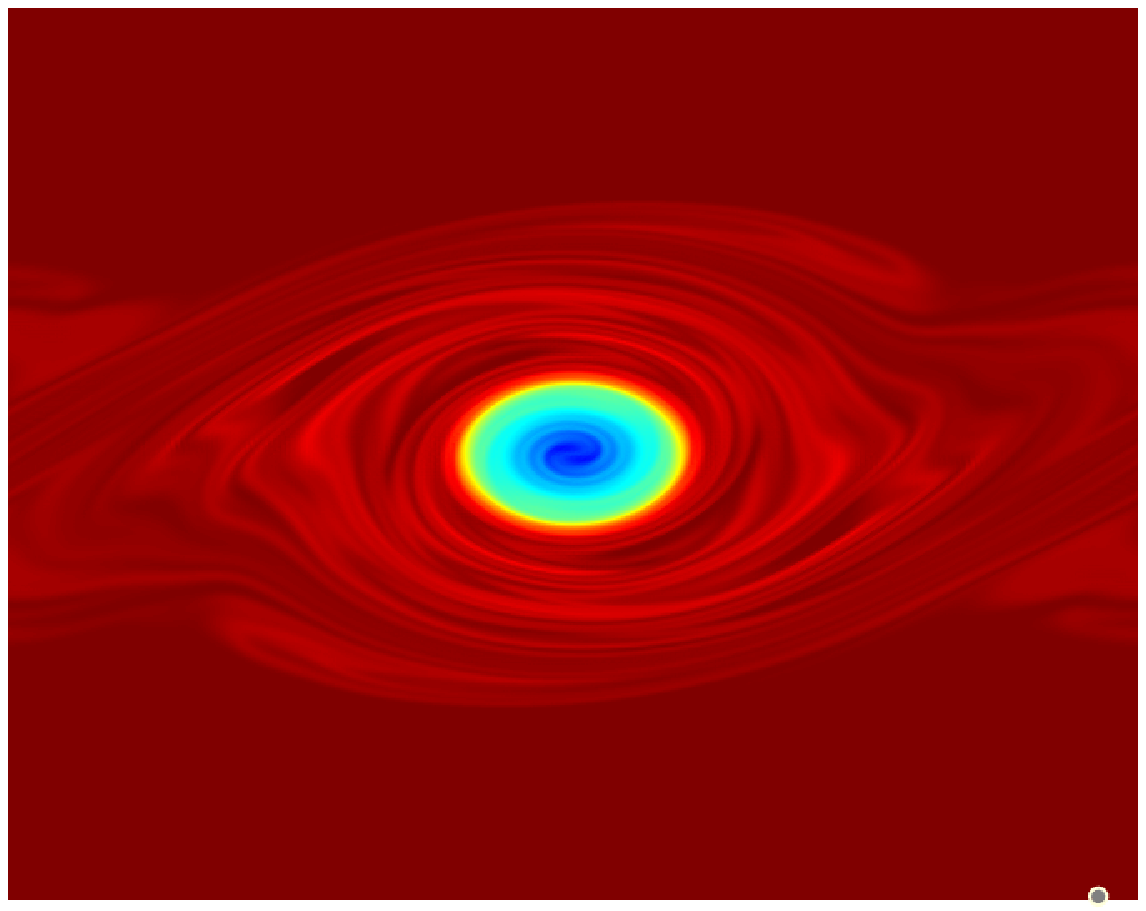}
  \caption{\label{fig:KH_vorticity}Example 2: From left to right and top to bottom, vorticity at $\bar{t}=5,10,17,34,56,200,272,308,400$ on the mesh with 131072 cells.}
\end{figure}

\begin{figure}
  \includegraphics[width=0.49\textwidth]{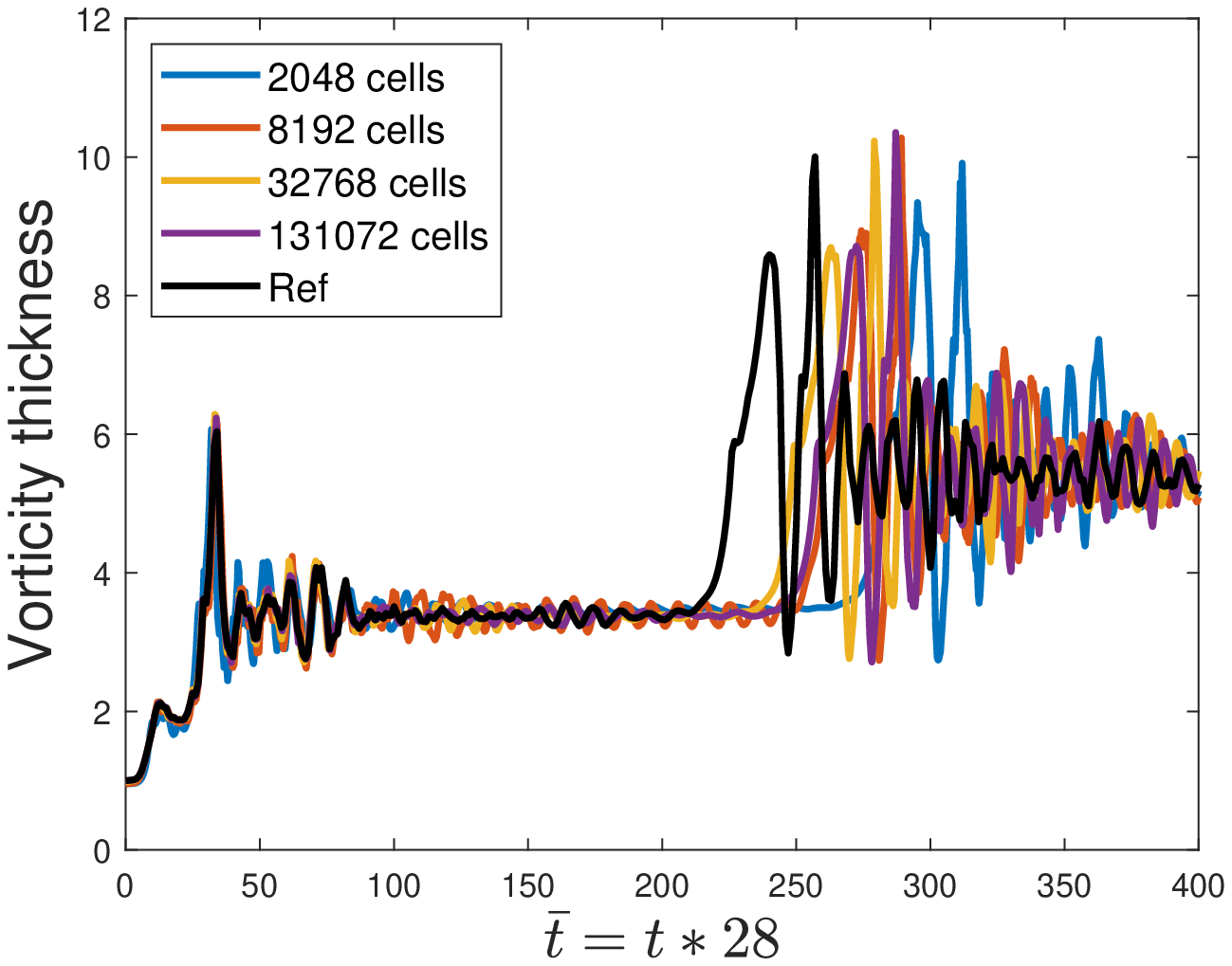}
  \hfill
  \includegraphics[width=0.49\textwidth]{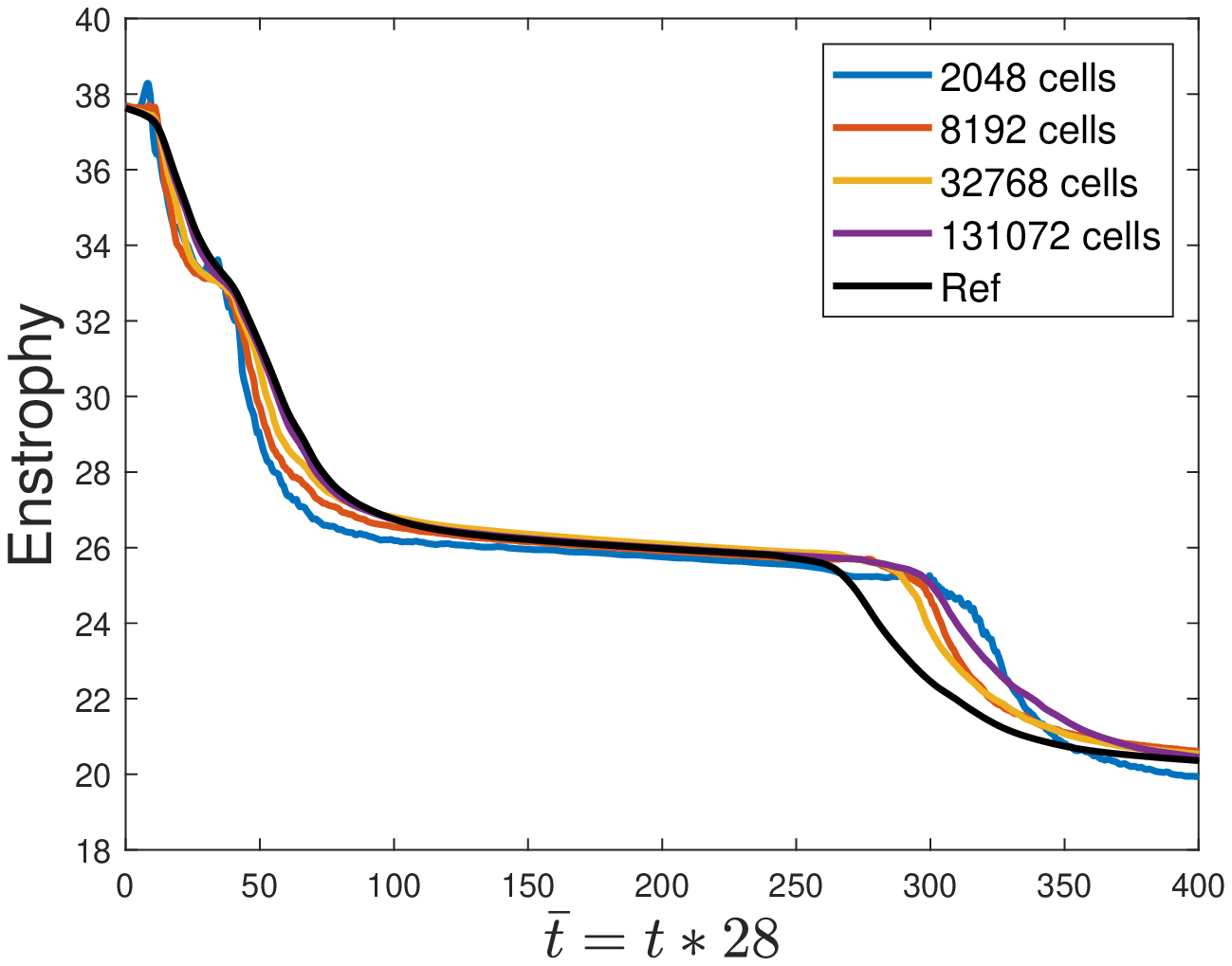}
  \caption{\label{fig:KH_QOI}Example 2: Vorticity thickness (left) and enstrophy (right).}
\end{figure}

\section{Conclusions and Outlook}
\label{sec:conclusions}

This paper successfully extends the Raviart--Thomas enriched
Scott--Vo\-ge\-lius elements from the stationary Stokes model problem
\cite{li2021low,JLMR2022} to the full instationary Navier--Stokes problem.
This results in a family of pressure-robust, divergence-free,
EMA-preserving, and convection-robust methods. However,
a stabilization for the Raviart--Thomas enrichment
part is required. Three choices are analyzed, among them a grad-div like stabilization
for the enrichment part and the upwind stabilization known from
DG framework, and all of them can be used and ensure convergence of at least order $k$
in the $L^2$-norm at any time. One important advantage of the novel family
is the possibility to reduce the scheme to a $\vecb{P}_k \times P_0$
scheme, which greatly reduces the number of unknowns and the
computational costs without compromising any
features listed above.
All theoretical results are supported by numerical experiments,
in particular by a simulation of the challenging two-dimensional
Kelvin--Helmholtz benchmark problem.

\smallskip
Out of the scope of this paper but
interesting aspects for future research
are the application of the new family
to real turbulent flows in three dimensions,
the investigation of preconditioners and iterative solvers and
additional convection stabilization, e.g.,
in the spirit of \cite{ABBGLM2021}, to improve the
order of convergence in convection-dominated
regimes.

\section*{Acknowledgments}
Naveed Ahmed would like to acknowledge financial support from the Gulf University
for Science and Technology for an internal Seed Grant (No.~278877).
Xu Li was supported by the China Scholarship Council (No.~202106220106)
and the National Natural Science Foundation of China (No.~12131014).
Christian Merdon gratefully acknowledges the funding by the German Science Foundation (DFG) within the project ``ME 4819/2-1".

\bibliographystyle{abbrv}
\bibliography{lit}

\end{document}